\documentclass[11pt,oneside,reqno]{article}

\usepackage{amsmath,amsthm,amssymb,color,bm}

\usepackage{array}
\usepackage[margin=2.5cm,a4paper]{geometry}
\usepackage{bbm}
\usepackage{paralist}
\usepackage{mathrsfs}
\usepackage[breaklinks=true]{hyperref}
\usepackage{color}

\usepackage[square]{natbib} 

\renewcommand{\cite}{\citep*}

\numberwithin{equation}{section}
\allowdisplaybreaks[4]

\theoremstyle{plain}
\newtheorem{theorem}{Theorem}[section]

\newtheorem{lemma}[theorem]{Lemma}

\theoremstyle{definition}

\newtheorem{remark}[theorem]{Remark}

\makeatletter

\edef\-#1{\noexpand\ifmmode {\noexpand\bar{#1}} \noexpand\else \-#1\noexpand\fi}

\renewcommand{\phi}{\varphi}
\newcommand{\eps}{\varepsilon}
\newcommand{\eq}{\eqref}

\newcommand{\dtv}{\mathop{d_{\mathrm{TV}}}}

\newcommand{\bigo}{\mathrm{O}}

\def\tsfrac#1#2{{\textstyle\frac{#1}{#2}}}

\newcommand{\I}{\mathrm{I}}

\newcommand{\MN}{\mathrm{MN}}

\newcommand{\IE}{\mathbbm{E}}
\newcommand{\IP}{\mathbbm{P}}

\newcommand{\Vol}{\mathop{\mathrm{Vol}}}
\newcommand{\e}{{\mathrm{e}}}

\newcommand{\law}{\mathscr{L}}
\newcommand{\eqlaw}{\stackrel{\mathscr{D}}{=}}

\newcommand{\IZ}{\mathbbm{Z}}
\newcommand{\IR}{\mathbbm{R}}

\newcommand{\Id}{\mathrm{Id}}

\def\be#1{\begin{equation*}#1\end{equation*}}
\def\ben#1{\begin{equation}#1\end{equation}}
\def\bes#1{\begin{equation*}\begin{split}#1\end{split}\end{equation*}}
\def\besn#1{\begin{equation}\begin{split}#1\end{split}\end{equation}}
\def\bg#1{\begin{gather*}#1\end{gather*}}

\def\ba#1{\begin{align*}#1\end{align*}}
\def\ban#1{\begin{align}#1\end{align}}

\def\klr#1{(#1)}
\def\bklr#1{\bigl(#1\bigr)}
\def\bbklr#1{\Bigl(#1\Bigr)}
\def\bbbklr#1{\biggl(#1\biggr)}

\def\bkle#1{\bigl[#1\bigr]}

\def\klg#1{\{#1\}}
\def\bklg#1{\bigl\{#1\bigr\}}

\def\norm#1{\Vert#1\Vert}

\def\bbnorm#1{\Bigl\Vert#1\Bigr\Vert}
\def\bbbnorm#1{\biggl\Vert#1\biggr\Vert}

\def\abs#1{\vert#1\vert}
\def\babs#1{\bigl\vert#1\bigr\vert}
\def\bbabs#1{\Bigl\vert#1\Bigr\vert}
\def\bbbabs#1{\biggl\vert#1\biggr\vert}
\def\bbbbabs#1{\Biggl\vert#1\Biggr\vert}

\def\bmid{\big\vert}

\renewcommand\section{\@startsection {section}{1}{\z@}%
{-3.5ex \@plus -1ex \@minus -.2ex}%
{1.3ex \@plus.2ex}%
{\center\small\sc\mathversion{bold}\MakeUppercase}}

\def\subsection#1{\@startsection {subsection}{2}{0pt}%
{-3.5ex \@plus -1ex \@minus -.2ex}%
{1ex \@plus.2ex}%
{\bf\mathversion{bold}}{#1}}

\def\subsubsection#1{\@startsection{subsubsection}{3}{0pt}%
{\medskipamount}%
{-10pt}%
{\normalsize\itshape}{\kern-2.2ex. #1.}}

\def\blfootnote{\xdef\@thefnmark{}\@footnotetext}

\makeatother

\def\s#1{^{(#1)}}

\newcommand\Dir{{\rm{Dir}}}
\newcommand\ta{\textnormal{\textbf{a}}}

\newcommand\tX{\textnormal{\textbf{X}}}
\newcommand\tZ{\textnormal{\textbf{Z}}}
\newcommand\tR{\textnormal{\textbf{R}}}
\newcommand\tx{\textnormal{\textbf{x}}}

\newcommand\te{\textnormal{\textbf{e}}}
\newcommand\tN{\textnormal{\textbf{N}}}
\newcommand\tB{\textnormal{\textbf{B}}}
\newcommand\ty{\textnormal{\textbf{y}}}
\newcommand\tq{\textnormal{\textbf{q}}}
\newcommand\tV{\textnormal{\textbf{V}}}

\newcommand{\teps}{\boldsymbol{\eps}}
\newcommand{\simp}{\Delta_{K}}
\def\cC{\mathcal{C}}

\def\Zx{\tZ_\tx}
\def\Zxt{\tZ_\tx(t)}

\def\e{\eps}
\def\tM{\textnormal{\textbf{M}}}
\def\tD{\textnormal{\textbf{D}}}
\def\tW{\textnormal{\textbf{W}}}
\def\tw{\textnormal{\textbf{w}}}
\def\tY{\textnormal{\textbf{Y}}}
\def\d{\delta}
\def\Bin{\text{Bin}}

\begin{document}

\title{\sc\bf\large\MakeUppercase{Dirichlet approximation of equilibrium distributions in Cannings models with mutation}}
\author{\sc Han~L.~Gan, Adrian R\"ollin, and Nathan Ross}
\date{\it Washington University in St.\ Louis, National University of Singapore, and University of Melbourne}
\maketitle

\begin{abstract} 
Consider a haploid population of fixed finite size with a finite number of allele types and having Cannings exchangeable genealogy with neutral mutation. The stationary distribution of the Markov chain of allele counts in each generation is an important quantity in population genetics but has no tractable description in general.
We provide upper bounds on the distributional distance between the Dirichlet distribution and this finite population stationary distribution for the Wright-Fisher genealogy with general mutation structure and the Cannings exchangeable  genealogy with parent independent mutation structure. In the first case, the bound is small if the population is large and the mutations do not depend too much on parent type; ``too much" is naturally quantified by our bound. In the second case, the bound is small if the population is large and the chance of three-mergers in the Cannings genealogy is small relative to the chance of two-mergers; this is the same condition to ensure convergence of the genealogy to Kingman's coalescent. These results follow from a new development of Stein's method for the Dirichlet distribution based on Barbour's generator approach and a probabilistic description of the semigroup of the Wright-Fisher diffusion due to Griffiths, and Li and Tavar\'e.
\end{abstract}


\section{Introduction}

We consider a neutral Cannings model with mutation in a haploid population of constant size~$N$ with~$K$ alleles. In each generation every individual has a random number of offspring such that the total number of offspring is~$N$. Different generations have i.i.d.\ offspring count vectors with distribution given by an exchangeable vector~$\tV:=(V_1, \ldots, V_N)$ not identically equal to~$(1, \ldots, 1)$;~$V_i$ is the number of offspring of individual~$1\leq i \leq N$. The random genealogy induced from this description is referred to as the \emph{Cannings model} \cite{Cannings1974}; particular instances are the \emph{Wright-Fisher model}, where~$\law(\tV)$ is multinomial with~$N$ trials and probabilities~$(1/N, \ldots, 1/N)$, and the \emph{Moran model}, where~$\law(\tV)$ is described by choosing a uniform pair of indices~$(I,J)$ and setting~$V_I=2$,~$V_J=0$ and~$V_i=1$ for~$i\not=I, J$. On top of the random genealogy given by~$\law(\tV)$, we put a mutation structure as follows. Given a child's parent type is~$i$, the child is of type~$j$ with probability~$p_{i j}$, where~$\sum_{j} p_{i j} = 1$. The type of each child in a given generation is chosen independently conditional on the genealogy of that generation and the parent's type. It is easy to see that this rule induces a time-homogeneous Markov chain~$(\tX(0), \tX(1), \ldots)$ with state space~$\{\tx\in\IZ_{\geq0}^{K-1}: \sum_{i=1}^{K-1} x_i\leq N\},$ where for~$i=1,\ldots, K-1$,~$X_i(n)$ is the number of individuals in the population having allele~$i$ at time~$n$; note that the count for allele~$K$ is given by~$N-\sum_{i=1}^{K-1} X_i(n)$.

Since~$\tX(n)$ is a Markov chain on a finite state space, it has a stationary distribution. But it is typically not possible to write down an expression for such a stationary distribution --- an important exception is the Wright-Fisher model with parent independent mutation (PIM), meaning~$p_{i j}$ does not depend on~$i$ for~$j\not=i$. In general, if the population size~$N\to\infty$, then under some weak conditions \cite{Mohle2000} (discussed in more detail below in Remark~\ref{rem1}) the Cannings genealogy viewed backwards in time converges to Kingman's coalescent \cite{Kingman1982, Kingman1982b, Kingman1982a} and the mutation structure on top of the coalescent has a nice Poisson process description. But even in this limit the stationary distribution (of now proportions of the~$K$ alleles)
is notoriously difficult to handle outside of the PIM case; see \cite{Griffiths1994} \cite{Bhaskar2012} for work on sampling under the stationary distributions and \cite{Ethier1992} for a probabilistic construction. Even if a formula in the limit were available, it is in any case important in population-mutation models to understand the difference between finite~$N$  likelihoods and those in the~$N\to\infty$ limit \cite{Bhaskar2014} \cite{Fu2006} \cite{Lessard2007, Lessard2010}  \cite{Mohle2004}.

Our approach to understanding these finite population stationary distributions is to determine when they are close to the Dirichlet distribution, which arises as the stationary limit in the PIM case (in this case the process converges in a suitable sense to the Wright-Fisher diffusion). 

In the next section we present our main results. First, we give two approximation theorems providing upper bounds on the distributional distance between the Dirichlet distribution and, for the first result, the finite population stationary distribution for the Wright-Fisher genealogy with general mutation structure and, for the second result, the Cannings exchangeable genealogy with parent independent mutation structure. Second, we discuss a new development of Stein's method for the Dirichlet distribution from which the first two results follow. 

\section{Main results}

Before stating our main results, we need some notation and definitions, as well as a short discussion regarding Lipschitz functions defined on open convex sets and their extension to the boundary.

Denote by~$\Dir(\ta)$ the Dirichlet distribution with parameters~$\ta=(a_1,\dots,a_K)$, where~$a_1>0, \ldots, a_K>0$, supported on the~$(K-1)$-dimensional open simplex, which we parameterize as
\be{
	\simp=\left\{\tx=(x_1,\ldots, x_{K-1}): x_1> 0, \ldots, x_{K-1} > 0,  \sum_{i=1}^{K-1} x_i < 1\right\}\subset \IR^{K-1}.
}
Denote by~$\-\Delta_K$ the closure of~$\Delta_K$.
On~$\simp$,~$\Dir(\ta)$ has density
\ben{\label{1}
	\psi_\ta(x_1,\dots,x_{K-1}) = \frac{\Gamma(s)}{\prod_{i=1}^K \Gamma(a_i)} \prod_{i=1}^K x_i^{a_i-1}, 
}
where~$s=\sum_{i=1}^K a_i$, and where we set~$x_K=1-\sum_{i=1}^{K-1} x_i$, as we shall often do in this paper whenever considering vectors taking values in~$\Delta_{K}$. 

\def\BC{\mathrm{BC}}
\def\CL{C_{\mathrm{L}}}
\def\Cb{C_{\mathrm{b}}}
Let~$U$ be an open subset of~$\IR^n$. For~$m\geq 1$, we denote by~$\BC^{m,1}(U)$ the set of bounded functions~$g:U\to\IR$ that have~$m$ bounded and continuous partial derivatives and whose~$m$-th partial derivatives are Lipschitz continuous. In line with this notation, we denote by~$\BC^{0,1}(U)$ the set of bounded functions that are Lipschitz continuous. 
We denote by~$\norm{g}_\infty$ the supremum norm of~$g$, and, if the~$k$-th partial derivatives of~$g$ exist, we let
\be{
	\abs{g}_k=\sup_{1\leq i_1,\dots,i_k\leq n} \bbnorm{\frac{\partial^k g}{\partial x_{i_1}\cdots\partial x_{i_k}}}_\infty
}
and 
\be{ 
\abs{g}_{k,1} = 
\sup_{1\leq i_1,\dots,i_k\leq n}\sup_{\tx,\ty\in U}
\bbbabs{
	\frac{\partial^k\bklr{g(\tx)-g(\ty)}}{\partial x_{i_1}\cdots\partial x_{i_k}}}\frac{1}{\norm{\tx-\ty}_1}.
}
Note that we use the~$L_1$-norm in our definition of the Lipschitz constant instead of the usual~$L_2$-norm. This is purely a matter of convenience, since the~$L_1$-norm shows up naturally in our proofs.

If~$g\in \BC^{m,1}(U)$ and~$U$ is convex, then all partial derivatives up to order~$m-1$ are Lipschitz continuous, too, and for any~$0\leq k \leq m-1$,
\ben{\label{2}
	\abs{g}_{k,1}	= \abs{g}_{k+1}.
}
As a result, if~$U$ is an open convex set, then any function~$g\in \BC^{m,1}(U)$ and all its partial derivatives up to order~$m$ can be extended continuously to a function~$\-g$ defined on the closure~$\-U$ in a unique way, and we have~$\norm{\-g}_\infty=\norm{g}_\infty$,~$\abs{\-g}_k=\abs{g}_k$ for~$1\leq k\leq m$ and~$\abs{\-g}_{m,1}=\abs{g}_{m,1}$. We can therefore identify the set of functions~$\BC^{m,1}(U)$ with set of  extended functions~$\BC^{m,1}(\-U)$.

\subsection{Wright-Fisher model with general mutation structure}

Our first result is a bound on the approximation of the stationary distribution
of the Wright-Fisher model with general mutation structure by
a Dirichlet distribution.

\begin{theorem}\label{THM1}
Let the~$(K-1)$-dimensional vector\/~$\tX$ be distributed as a stationary distribution of the Wright-Fisher model for a population of~$N$ haploid individuals with~$K$ types and mutation
structure~$p_{i j}$,~$1\leq i,j\leq K$; set~$\tW=\tX/N$. Let\/~$\ta$ be a~$K$-vector of positive numbers, set~$s=\sum_i a_i$, and let\/~$\tZ\sim \Dir(\ta)$. Then, for any~$h\in \BC^{2,1}(\-\Delta_K)$,
\be{
\left|\IE h(\tW)-\IE h(\tZ) \right| \leq \frac{\abs{h}_1}{s} A_1+\frac{\abs{h}_2}{2(s+1)} A_2 +  \frac{\abs{h}_{2,1}}{18(s+2)} A_3,
}
where
\bg{
A_1 =  2N(K+1)\tau,\quad
A_2= N K^2 \mu^2 + 2K\mu, \quad
A_3= 8 N K^3 \mu^3 +\frac{16\sqrt{2}K^3}{N^{1/2}},
}
with
\ben{
	\tau = \sum_{i=1}^K\sum_{\substack{j=1\\j\neq i}}^K\,\bbabs{\,p_{ij}-\frac{a_j}{2N}},\qquad \mu = \sum_{i=1}^K\sum_{\substack{j=1\\j\neq i}}^K p_{ij}.
}
Moreover, there is a constant~$C=C(\ta)$ such that
\be{
	\sup_{A\in \cC_{K-1}}\babs{\IP[\tW\in A] - \IP[\tZ\in A]}\leq C\bklr{A_1+A_2+A_3}^{\theta/(3+\theta)},
}
where~$\cC_{K-1}$ is the family of convex sets on\/~$\IR^{K-1}$ and where 
$\theta=\theta(\ta)>0$ is given at~\eq{10}.
\end{theorem}

\begin{remark}
To interpret the bounds of the theorem, if
$p_{i j}=\frac{a_j}{2N}+\eps_{ij}$ for~$i\not=j$, and we assume~$\abs{\eps_{ij}}\leq \eps$, then
\ba{
\tau \leq K(K-1)\eps
\qquad
\mu \leq \frac{(K-1)s}{2N}+K(K-1)\eps,
}
so that for fixed~$K$ and~$\ta$ (though note that, for smooth functions, the reliance on these parameters is explicit),
\be{
A_1 = \bigo(N\eps), \qquad A_2 = \bigo(N^{-1}+N \eps^2), \qquad A_3 = \bigo(N^{-1/2}+N \eps^3).
}
In particular, in the PIM case, where~$\eps=0$, our bound on smooth functions is of order~$N^{-1/2}$, and for the convex set metric of order~$N^{-1/8}$ if~$\min\{a_1,\dots,a_K\}\geq 1$ and otherwise the order of the bound is some negative power of~$N$ having
a more complicated relationship to~$\ta$, but which is easily read from~\eq{10}.
In the special case where~$K=2$,~$\eps=0$ and~$h$ has six bounded derivatives, \cite{Ethier1977} derived a bound analogous to that of Theorem~\ref{THM1}, but of order
$N^{-1}$.
In the general case our bound quantifies the effect of non-PIM: if
$N\eps \to 0$ as~$N\to\infty$ then the stationary distribution converges to the Dirichlet distribution.
\end{remark}

\subsection{Cannings model with parent-independent mutation structure}

Our next result is for the general Cannings exchangeable non-degenerate genealogy. The bounds
are in terms of the moments of the offspring vector~$\tV$; hence, let
\ban{\label{3}
	\alpha:=\IE\klg{V_1 (V_1-1)},
	\quad
	\beta:=\IE\klg{ V_1 (V_1-1)(V_1-2)},
	\quad
 	\gamma:=\IE \klg{V_1 (V_1-1)V_2(V_2-1)}  
}
(and note that these quantities depend on~$N$).

\begin{theorem}\label{THM2}
Let the $(K-1)$-dimensional vector~$\tX$ be a stationary distribution of the Cannings model for a population of size~$N\geq 4$ with non-degenerate exchangeable genealogy~$\law(\tV)$. Assume we have
parent independent mutation structure; that is,~$p_{ij}=\pi_j$,~$1\leq i\not= j \leq K$,~ for some~$\pi_1,\dots,\pi_K>0$,
and $p_{ii}=1-\sum_{j\not=i} \pi_j$.
Let~$\alpha$,~$\beta$, and~$\gamma$ be as defined at~\eq{3},
and for~$i=1,\ldots, K$, set~$a_i=\frac{2(N-1)\pi_i}{\alpha}$ and~$s=\sum_i a_i$. Let
$\tW=\tX/N$, and let~$\tZ\sim \Dir(\ta)$. Then, for any~$h\in\BC^{2,1}(\-\Delta_K)$,
\be{
\left|\IE h(\tW)-\IE h(\tZ) \right| \leq \frac{\abs{h}_2}{2(s+1)} A_2 +  \frac{\abs{h}_{2,1}}{18(s+2)} A_3,
}
where, with~$\eta=N\alpha^{-1}\sum_{j=1}^K\pi_j=\frac{sN}{2(N-1)}$,
\ba{
A_2&= \bbklr{\frac{\alpha}{N}}^2\eta^2 K^2 + \frac{\alpha}{N}\bbklr{\eta^2 (K^2+1)  + 2\eta K^2} + \frac{3\eta K}{N} ,\\
A_3	&= 2K^3\left(1+\eta\sqrt{\frac{\alpha}{N}}+\sqrt{\frac{\eta}{N}} \right)\left( 
	\eta\bbbklr{\frac{\alpha}{N}}^{3/4} 
	+ \left( \frac{12\beta}{\alpha N}+ \frac{ 24 \gamma}{\alpha N}\right)^{1/4} 
	+\frac{1}{N^{1/2}}\left(3\eta^2\frac{\alpha}{N} + \frac{\eta}{N} \right)^{1/4} 
	 \right)^2.
}
Moreover, there is a constant~$C=C(\ta)$ such that
\ben{
	\sup_{A\in \cC_{K-1}}\babs{\IP[\tW\in A] - \IP[\tZ\in A]}\leq C\bklr{A_2+A_3}^{\theta/(3+\theta)},
}
where~$\cC_{K-1}$ is the family of convex sets on\/~$\IR^{K-1}$ and where 
$\theta=\theta(\ta)>0$ is given at~\eq{10}.
\end{theorem}
\begin{remark}\label{rem1}
To interpret the bound of the theorem, we note that the bound goes to zero if, as~$N\to\infty$, $\eta\leq s$ remains bounded and all three of 
\ban{\label{4}
&\frac{\alpha}{N}, \qquad  \frac{\beta}{\alpha N},   \qquad
\frac{\gamma}{\alpha N},
}
tend to zero. And for the convergence to be to non-degenerate, we must have
\ben{\label{5}
a_i = \frac{2(N-1)\pi_i}{\alpha}\to \tilde{a}_i,
}
for some limiting positive~$\tilde{a}_i$,~$i=1,\ldots, K$, which also implies that 
$\eta=sN/(2(N-1))$ converges to a positive constant. 
As briefly mentioned above, under appropriate assumptions, the exchangeable genealogy alone (that is, without mutation structure)
converges to Kingman's coalescent as~$N\to\infty$. This convergence occurs if and only if
\cite{Mohle2000} \cite{Mohle2001, Mohle2003} 
\ben{\label{6}
\frac{\beta}{ \alpha N}\to 0.
}
In this case and also assuming a limiting scaling of the mutation 
probabilities given by~\eq{5},
 the finite population stationary distribution converges to the
stationary distribution of a Wright-Fisher diffusion, that is,~$\Dir(\ta)$. 
At first glance it appears that demanding the terms of~\eq{4} tend to zero
is a stronger requirement for convergence than the sufficient~\eq{6},
but \cite[Lemma~5.5]{Mohle2003}, \cite[Display~(16)]{Mohle2000}
show that~\eq{6} also implies
\be{
\frac{\alpha}{N}\to 0\,  \qquad \text{ and   } \qquad \frac{\gamma}{\alpha N}\to 0.
}
So, in fact, our bound goes to zero assuming only~\eq{6} and thus quantifies the convergence of the stationary distribution in terms of natural quantities. Assuming~$\eta$ remains bounded, we obtain
\be{
	A_2 = \bigo\bbbklr{K^2\frac{\alpha}{N}+\frac{K}{N}},\qquad A_3 = \bigo\bbbklr{K^3\bbklr{\frac{\alpha}{N}}^{3/2} + K^3\bbklr{\frac{\beta}{\alpha N}}^{1/2}+K^3\bbklr{\frac{\gamma}{\alpha N}}^{1/2}+\frac{K^3}{N}}.
}
\end{remark}

\begin{remark}
For the stationary distribution of types in an exchangeable Cannings genealogy with general mutation structure,
 a bound with features similar to those of Theorems~\ref{THM1} and~\ref{THM2} should be possible using our methods. However,
 the formulation and proof of such a result would be rather messy, and so, for the sake of exposition and clarity, we 
 present two separate theorems to handle more specific situations.
\end{remark}

\subsection{Stein's method of exchangeable pairs for the Dirichlet distribution}
 
Theorems~\ref{THM1} and~\ref{THM2} follow from a new development of Stein's method for the Dirichlet distribution.
Stein's method  \cite{Stein1972, Stein1986} is a powerful tool for providing bounds on the approximation of
a probability distribution of interest by a well understood target distribution; see \cite{Chen2011} and \cite{Ross2011}
for recent introductions, and \cite{Chatterjee2014} for a recent literature survey. We show a general exchangeable pairs Dirichlet approximation theorem very much in the spirit of exchangeable pairs approximation results for other distributions; e.g., normal \cite[Theorem~1.1]{Rinott1997}; multivariate normal \cite[Theorem~2.3]{Chatterjee2008} \cite[Theorem~2.1]{Reinert2009};
exponential \cite[Theorem~1.1]{Chatterjee2011} \cite[Theorem~1.1]{Fulman2013}; beta \cite[Theorem~4.4]{Dobler2015}; limits in Curie-Weiss models \cite[Theorem~1.1]{Chatterjee2011a}. In what follows, sums range from~$1$ to~$K$ unless otherwise stated.

\begin{theorem}\label{THM3}
Let~$\ta=(a_1,\ldots, a_K)$ be a vector of positive numbers and set~$s=\sum_{i=1}^K a_i$.
Let~$(\tW,\tW')$ be an exchangeable pair of~$(K-1)$ dimensional random vectors with non-negative entries with sum no greater than one.
Also let~$\Lambda$ be an invertible matrix and~$\tR$ be a random vector such that
\ben{
\IE [\tW'-\tW| \tW]=\Lambda (\ta-s \tW) + \tR. \label{7}
}
Then, for any~$h\in\BC^{2,1}(\-\Delta_K)$,
\ben{\label{8}
\abs{\IE h(\tW)-\IE h(\tZ)}\leq \frac{\abs{h}_1}{s} A_1 + \frac{\abs{h}_2}{2(s+1)} A_2 +\frac{\abs{h}_{2,1}}{6(s+2)} A_3,
}
where 
\ba{
A_1&:=\sum_{m,i} \babs{(\Lambda^{-1})_{i,m}}\IE \abs{R_m },\\ 
A_2&:=  \sum_{m,i,j} \babs{(\Lambda^{-1})_{i,m}} \IE \left| \Lambda_{m,i} W_i(\delta_{i j}-W_j)-\frac{1}{2}\IE[ (W_m'-W_m)(W_j'-W_j)|\tW] \right|, \\
A_3&:=\sum_{m,i,j,k} \babs{(\Lambda^{-1})_{i,m}} \IE \left|(W_m'-W_m)(W_j'-W_j)(W_k'-W_k) \right|.
}
Moreover, there exists a constant~$C=C(\ta)$ such that
\ben{\label{9}
	\sup_{A\in \cC_{K-1}}\babs{\IP[\tW\in A] - \IP[\tZ\in A]}\leq C\bklr{A_1+A_2+A_3}^{\theta/(3+\theta)},
}
where~$\cC_{K-1}$ is the family of convex sets on\/~$\IR^{K-1}$ and
\ben{\label{10}
	\theta = \frac{\theta_\wedge}{\theta_\wedge+\theta_\circ},\qquad \theta_\wedge = 1\wedge\min\{a_1,\dots,a_K\},\qquad \theta_\circ=\sum_{i=1}^K\bklr{1-1\wedge a_i}.
}
Additionally, if~$\Lambda$ is a multiple of the identity matrix, then
the result still holds assuming only that~$\law(\tW)=\law(\tW')$, in which 
case the factor~$\frac{\abs{h}_{2,1}}{6(s+2)}$ in \eq{8} can be improved 
to~$\frac{\abs{h}_{2,1}}{18(s+2)}$
\end{theorem}

The layout of the remainder of the paper is as follows. We finish the introduction by applying
 Theorem~\ref{THM3} in an easy example, the multi-colored P\'olya urn. 
In Section~\ref{sec1} we develop Stein's method for the Dirichlet distribution and
prove Theorem~\ref{THM3}.
In Section~\ref{sec2} we prove Theorem~\ref{THM1}, the bounds for the Wright-Fisher model and
in Section~\ref{sec3} we prove Theorem~\ref{THM2}, the bounds for the PIM Cannings model.

\paragraph{A simple example: Multi-colored P\'olya Urn.}

In order to illustrate how Theorem~\ref{THM3} is applied, we use it to bound the error in approximating the counts in the classical P\'olya urn by a Dirichlet distribution. The result is new to us, but a bound in the
Wasserstein distance could be obtained 
from analogous bounds for the beta distribution \cite{Goldstein2013} \cite{Dobler2015} using
the iterative urn approach of \cite{Pekoz2014a}.

An urn initially contains~$a_i>0$ ``balls" of color~$i$ for~$i =1,2,\ldots, K$ with a total number of balls~$s= \sum_{i=1}^K a_i$. At each time step, draw a ball uniformly at random from the urn and replace it along with another ball of the same color. 
Let~$\tX(n) = (X_1(n), X_2(n), \ldots , X_{K-1}(n))$, where~$X_i(n)$ is the number of times color~$i$ was drawn up to and including the~$n$th draw. It is well known (see, e.g., \cite{Mahmoud2009}) that as~$n \to \infty$, 
\be{  
\tW(n):= \frac{\tX(n)}{n} \stackrel{d}{\longrightarrow} \Dir(\ta),
}
and we provide a bound on the approximation of the distribution of~$\tW(n)$ by the Dirichlet limit.

\begin{theorem}\label{THM4}
Let~$\ta=(a_1, \ldots, a_K)$ be a vector of positive numbers,~$s=\sum_{i=1}^K{a_i}$,~$\tZ\sim\Dir(\ta)$, and~$\tW(n)$ be the P\'olya urn proportions as defined above. Then, for any~$h\in\BC^{2,1}(\-\Delta_K)$, 
\be{
\abs{\IE h(\tW(n)) -\IE h(\tZ)}\leq \frac{s}{n(s+1)} \abs{h}_2 +  \frac{(K-1)(3K-5) (n+s-1)}{18n^2(s+2)} \abs{h}_{2,1}.
}
Moreover, there exists a constant~$C=C(\ta)$ such that
\be{
\sup_{A\in \cC_{K-1}}\abs{\IP[\tW\in A] - \IP[\tZ\in A]}\leq Cn^{-\theta/(3+\theta)},
}
where $\theta=\theta(\ta)>0$ is defined at~\eq{10}.
\end{theorem}

We use Theorem~\ref{THM3} to prove the result. To define the exchangeable pair,
note that we can set~$\tX(n) = \sum_{j=1}^n \tY(j)$ where, for~$\te_i$ equal to the~$i$th unit vector,~$\tY(j)=\te_i$ if color~$i$ is drawn on the~$j$th draw.
It is easy to check that 
\be{
\IP\bkle{\tY(j) = \te_i \bmid \tX(j-1)} = \frac{X_i(j-1) + a_i}{j-1+s}.
}
We define the exchangeable pair~$(\tW, \tW')$ (dropping the~$n$ to ease notation) by resampling the last draw~$\tY(n)$; that is, 
\be{
\tW' = \tW - \frac{\tY(n)}{n} + \frac{\tY'(n)}{n},
}
where conditional on~$\tX(n-1)$,~$\tY'(n)$ and~$\tY(n)$ are i.i.d.\ Before computing the terms appearing in the bound of Theorem~\ref{THM3}, we record a lemma.
\begin{lemma}\label{lem1}
Recalling the notation and definitions above, and let~$\d_{ij}$ denote the Kronecker delta function,
\ba{
\IE(\tY'(n) | \tW) &= \frac{1}{n+s-1}\left[ \ta + (n-1)\tW\right],\\
\IE(Y_i'(n)Y_j(n) | \tW) &= \frac{1}{n+s-1} \IE \left[ a_i W_j + nW_iW_j - W_i \d_{ij}\right].
}
\end{lemma}
\begin{proof}
First note that 
\be{
\IE[\tY'(n)|  (\tY(1), \ldots, \tY(n))]=\frac{1}{n+s-1} \left[ \ta + n\tW - \tY(n) \right].
}
The first equality now follows by taking expectation conditional on~$\tW$ and noting that exchangeability implies~$\IE [\tY(n)| \tW]=\tW$.
For the second identity, use the previous display to find
\ba{
\IE[ Y_i'(n) Y_j(n) | Y_j(1), \ldots Y_j(n)] = \frac{Y_j(n)}{n+s-1} \left[ a_i +n W_i - Y_i(n) \right],
}
and taking expectation conditional on~$\tW$, noting~$\IE \tY(n)=\tW$ and~$Y_i(n) Y_j(n)=\delta_{i j} Y_i(n)$,
yields
\be{
\IE(Y_i'(n)Y_j(n)| \tW)= \frac{1}{n+s-1} \IE \left[ a_i W_j + nW_iW_j - W_i \d_{ij}\right].   \qedhere
} 
\end{proof}

\begin{proof}[Proof of Theorem~\ref{THM4}]
We apply Theorem~\ref{THM3} with the exchangeable pair defined above. We show below
that for~$i,j,k\in\{1,\ldots, K\}$, 
\ban{
&\IE[\tW'-\tW|\tW]= \frac{1}{n(n+s-1)}\left(\ta - s \tW\right), \label{12}\\
&\IE[(W_i' - W_i)(W_j' - W_j)| \tW]  =\frac{\d_{ij}(a_i + (2n+s)W_i) - a_i W_j - a_j W_i - 2nW_iW_j}{n^2(n+s-1)}, \label{13}\\
&\IE \left| (W_i'-W_i)(W_j'-W_j)(W_k'-W_k) \right| \leq n^{-3}(1-\I[\text{$i,j,k$ distinct}]), \label{14}
}
so we can apply the Theorem~\ref{THM3} with~$\Lambda=\frac{1}{n (n+s-1)} \times \Id$.
In this case, using~\eq{13}, 
\ba{
A_2&= n(n+s-1)\sum_{i,j=1}^{K-1}  \IE \left| \frac{1}{n(n+s-1)} W_i(\d_{ij}-W_j) - \frac{1}{2} \IE\left[(W_i'-W_i)(W_j'-W_j)|\tW\right]\right|\\
	&=\frac{1}{2n} \sum_{i,j=1}^{K-1} \IE \left| \d_{ij}(a_i + sW_i) - a_iW_j - a_jW_i\right|\leq \frac{2s}{n}.
}
Now using~\eq{14}, we have
\be{
A_3\leq \frac{(K-1)(3K-5) (n+s-1)}{n^2}.
}
Finally the form of~\eq{12} makes it clear that~$R=0$ and so~$A_1=0$. Putting together the last two displays yields the result.

All that is left is to show~\eq{12},~\eq{13}, and~\eq{14}. Lemma~\ref{lem1} implies
\be{
\IE[\tW' - \tW | \tW]= \frac{1}{n} \IE [ \tY'(n) - \tY(n) | \tW ]=\frac{1}{n(n+s-1)}\left(\ta - s \tW\right),
}
which is~\eq{12}. For~\eq{13}, use Lemma~\ref{lem1} and that~$Y_i(n) Y_j(n)=\delta_{i j} Y_i(n)$ and~$Y_i'(n) Y_j'(n)=\delta_{i j} Y_i'(n)$ to find
\ba{
&\IE[(W_i' - W_i)(W_j' - W_j)| \tW] = \frac{1}{n^2} \IE[ Y_i'(n)Y_j'(n) + Y_i(n)Y_j(n) - Y_i'(n)Y_j(n) - Y_i(n)Y_j'(n) | \tW]\\
	&= \frac{1}{n^2} \left[\frac{\d_{ij}}{n+s-1}(a_i + (n-1)W_i) + \d_{ij}W_i - \frac{1}{n+s-1}  \left(a_i W_j + a_j W_i + 2nW_iW_j - 2 W_i \d_{ij}\right) \right]\\
	&= \frac{1}{n^2(n+s-1)} \left[ \d_{ij}(a_i + (2n+s)W_i) - a_i W_j - a_j W_i - 2nW_iW_j\right].
}
Finally~\eq{14} follows noting that~$\abs{W_i'-W_i}\leq 1/n$ and that at most two of the~$(W_i'-W_i)$ can be non-zero.

The bound on the convex set metric is immediate from the bounds on~$A_2$ and~$A_3$ and \eq{9}.
\end{proof}

\section{Stein's~method~for~the~Dirichlet~distribution}\label{sec1}

\subsection{Stein operator}

In order to apply Stein's method we need a characterizing operator for the Dirichlet distribution, which is provided below. 
Let~$\delta_{ij}$ denote the Kronecker delta function, and for a function~$f$, let~$f_j$ be the partial derivative of~$f$ with respect to the~$j$th component,~$f_{i j}$ the 2nd partial derivative, and so on.

\begin{lemma}
Let~$a_1, \ldots, a_K$ be positive numbers and~$s=\sum_{i=1}^K a_i$. The random vector~$\tW\in\simp$ has distribution~$\Dir (a_1,\ldots,a_K)$ if and only if for all~$f\in\BC^{2,1}(\Delta_K)$
\be{
\IE\left[ \sum_{i,j=1}^{K-1} W_i(\delta_{i j}-W_j)f_{i j}(\tW)
+\sum_{i=1}^{K-1}(a_i-s W_i)f_{i}(\tW)\right]=0.
}
\end{lemma}
The forward implication of the lemma is straightforward 
and the backwards follows by taking expectations against polynomials~$f$ to yield formulas for mixed moments of~$\tW$. Also note that  
\ben{\label{15}
\mathcal{A}f(\tx) := \frac{1}{2}\left[\sum_{i,j=1}^{K-1} x_i(\delta_{i j}-x_j) f_{i j}(\tx)
+\sum_{i=1}^{K-1}(a_i-s x_i)f_i(\tx)\right]
}
is the generator of the Wright-Fisher diffusion which has the Dirichlet as its unique stationary distribution; see \cite{Wright1949}, \cite{Ethier1976}, \cite{Shiga1981}.

\subsection{Bounds on the solution to the Stein equation}

To apply Stein's method, we proceed as follows. Let~$\tZ\sim\Dir(\ta)$, and let~$h:\-\Delta_K\to\IR$ be some measurable test function. If~$h$ is bounded, then clearly~$\IE \abs{h(\tZ)} < \infty$. Assume we have a function~$f:=f_h$ that solves
\ben{\label{16}
\sum_{i,j=1}^{K-1} x_i(\delta_{i j}-x_j)f_{i j}(\tx)
+\sum_{i=1}^{K-1}(a_i-s x_i) f_i (\tx)=h(\tx)-\IE h(\tZ) =: \tilde{h}(\tx),
}
and note that replacing~$\tx$ by~$\tW$ in this equation and taking expectation gives an expression for~$\IE h(\tW)-\IE h(\tZ)$
in terms of just~$\tW$ and~$f$. 
Since this operator is twice the generator of the Wright-Fisher diffusion given by~\eq{15}, we use the generator approach of \cite{Barbour1990} \cite{Gotze1991},
and observe we may set~$f$ to be
\ben{\label{17}
 f(\tx) = - \frac{1}{2}\int_0^\infty \IE \left[ h(\Zxt) - h(\tZ) \right] dt = - \frac{1}{2}\int_0^\infty \IE\tilde{h}(\Zxt)  dt,
\qquad\tx\in\Delta_K, }
where~$(\Zxt)_{t\geq0}$ is the Wright-Fisher diffusion, defined by the generator~$\mathcal{A}$ with~$\tZ_\tx(0) = \tx$ (the factor 
of~$1/2$ in the expression appears since~\eq{15} is twice the generator of~$\tZ$).  
Using a probabilistic description of the Wright-Fisher semigroup due to 
\cite{Griffiths1983} and \cite{Tavare1984} we show that the above integral is well defined, and we obtain the following bounds on the solution~\eq{17} to the Stein equation~\eq{16}.

\begin{theorem}\label{THM5} 
If~$h:\-\Delta_K\to\IR$ is continuous, then~$f$ defined by \eq{17} is twice partially differentiable and solves \eq{16} for all~$\tx\in\Delta_K$, and we have the bound
\ben{\label{18}
\norm{f}_\infty\leq \frac{(s+1)}{s}\norm{\tilde h}_\infty.
}
If~$h\in\BC^{m,1}(\-\Delta_K)$ for some~$m\geq 0$, then~$f\in\BC^{m,1}(\-\Delta_K)$, and we have the bounds
\ben{\label{19}
\abs{f}_k\leq \frac{\abs{h}_k}{k(s+k-1)},\quad 1\leq k\leq m,\qquad\text{and}
\qquad
\abs{f}_{m,1}\leq \frac{\abs{h}_{m,1}}{m(s+m-1)}.
}
If~$m\geq 2$, then 
equation \eq{16} holds for all~$\tx\in\-\Delta_K$.
\end{theorem}

\begin{remark}
The Dirichlet distribution is a multivariate generalization of the Beta distribution for which Stein's method 
has recently been developed \cite{Dobler2012, Dobler2015} \cite{Goldstein2013} where bounds are derived
for the~$K=2$ case of the Stein equation used here. Direct comparisons are difficult in general since 
typically different derivatives of the test function appear. However one easily comparable bound is \cite[Proposition~4.2(b)]{Dobler2015} that~$|f|_1 \leq |h|_1/s$, which is the same as our bound in this case. 
In general, the bounds appearing in these other works are quite complicated, involving different expressions for different regions of the parameter space, whereas 
our bounds are very clean and have a simple relationship to the parameters. Furthermore our bounds apply in the multivariate setting.
\end{remark}

\begin{proof}[Proof of Theorem~\ref{THM5}] Throughout the proof we make the simplifying assumption that~$\IE h(\tZ)=0$ so that~$\tilde h = h$.
Following the generator approach of \cite{Barbour1990}, \cite{Gotze1991}; see also \cite[Appendix~B]{Gorham2016}; let~$\tx\in\-\Delta_K$, 
let~$(\Zxt)_{t\geq0}$ be the Wright-Fisher diffusion defined by the generator~$\mathcal{A}$ with~$\tZ_\tx(0) = \tx$,
and let~$f$ be as defined in~\eq{17}.

\smallskip
\noindent{\bf Construction of semigroup.}
The key to our bounds is a construction of the marginal variable~$\Zxt$ from \cite{Griffiths1983} \cite{Tavare1984}; see also the introduction of~\cite{Barbour2000}.
Let~$L_t$ be a pure death process on~$\{0, 1, \ldots\}\cup\{\infty\}$ started at~$\infty$ with death rates 
\ben{\label{20}
q_{i, i-1}= \frac{1}{2} i(i-1+s).
}
Denote by~$\MN_K(n; p_1,\dots,p_K)$ the~$K$-dimensional multinomial distribution with~$n$ trials and probabilities~$p_1,\dots,p_K$; by slight misuse of notation, we write~$\MN_K(L_t;\tx,x_K)$ to be 
short for~$\MN_K(L_t;x_1,\dots,x_{K-1},x_K)$. Conditional on~$L_t$, let~$\tN\sim\MN_K(L_t;\tx,x_K)$, where~$x_K=1-\sum_{i=1}^{K-1}x_i$.
Then,
\be{
	\law\bklr{\Zxt\bmid L_t,\tN} \sim \Dir(\ta+\tN).
}

\smallskip
\noindent{\bf Existence of solution to Stein equation on~$\boldsymbol{\Delta_K}$ and bound (\ref{18}).}
For~$n\geq 1$, let~$Y_n$ be the time the process~$L_t$ spends in state~$n$ and note that~$Y_n$ is exponentially distributed with rate~$n(n-1+s)/2$. Since
\be{
	\sum_{n\geq 1} \IE Y_n = \sum_{n\geq 1}\frac{2}{n(n+s-1)} \leq  \frac{2(s+1)}{s},
}
the random variable~$T = \inf\{t>0\,:\,L_t=0\} = \sum_{n\geq 1}Y_n$ is finite almost surely and has finite expectation. Observing that~$\IE\bklr{{ h}(\Zxt)\big| L_t=0}=0$ since, given~$L_t=0$, we have~$\Zxt \sim \Dir(\ta)$, it follows that
\besn{\label{457}
\int_0^\infty \babs{\IE\bklr{ h(\Zxt) }} dt &\leq \int_0^\infty \norm{h}_\infty \IP(L_t>0) dt \\
	&\leq \norm{h}_\infty \int_0^\infty \IP(T>t) dt =\norm{ h}_\infty\IE T < \infty.
}
Thus,
$f$ in \eq{17} is well-defined. 

To show that $f$ is in the domain of~$\cal{A}$ and satisfies $\mathcal{A}f= h$ under the assumption that~$h\in\BC(\-\Delta_K)$, the Banach space of bounded and continuous functions equipped with sup-norm, we follow the argument of \cite[Pages~301-2]{Barbour1990} also used in \cite[Appendix~B]{Gorham2016}. First,
\cite[Theorem~1]{Ethier1976} implies that
the semigroup $(T_t)_{t\geq0}$ defined by $T_t g(\tx)=\IE g(\Zxt)$ for $g\in\BC(\-\Delta_K)$ is strongly continuous. Note also that $\BC^{m,1}(\-\Delta_K)\subset\BC(\-\Delta_K)$ for all $m\geq 0$. 
We can therefore apply
\cite[Proposition~1.5(a), Page~9]{Ethier1986}, which implies that $f\s u(\tx):=- \frac{1}{2}\int_0^u \IE{h}(\Zxt)  dt$ is in the domain of~$\mathcal{A}$ and satisfies
\be{
\mathcal{A} f\s u (\tx)=h(\tx)-\IE h(\Zx(u)).
}
Furthermore, \cite[Corollary~1.6, Page~10]{Ethier1986} implies that $\mathcal{A}$ is a closed operator, 
so it is enough to show that as $u\to\infty$,
\ben{\label{458}
\norm{f\s u-f}_\infty\to0 \,\,\, \mbox{ and } \,\,\, \norm{\mathcal{A}f\s u-h}_{\infty}\to 0.
}
By definitions,~\eq{458} is implied by
\be{
\sup_{\tx\in\-\Delta_K}\int_u^\infty \babs{\IE\bklr{ h(\Zxt) }} dt\to 0 \,\,\, \mbox{ and } \,\,\, \sup_{\tx\in\-\Delta_K}\IE  h(\Zx(u))\to 0,
}
as $u\to\infty$. But the first limit follows from~\eq{457} and the second is because
\be{
\sup_{\tx\in\-\Delta_K}\IE  h(\Zx(u))\leq \norm{h}_\infty \sup_{\tx\in\-\Delta_K} \dtv\bklr{\law(\Zx(u)), \Dir(\ta)} \leq \norm{h}_\infty \IP(L_u>0)\to 0.
}

The boundedness of the solution follows essentially from the computations above, but we give a slightly different argument 
in detail, since a similar but more complicated one is used later. Compute
\bes{
	-2f(\tx)
	& =\int_0^\infty \IE  h(\Zxt) dt 
	 =\IE \int_0^\infty \IE \bklr{h(\Zxt)\big|L_t} dt \\
	& =\IE \int_0^\infty \sum_{n\geq1}  \IE\bklr{h(\Zxt)\big| L_t=n}\I[L_t=n] dt\\
	 &=\IE \sum_{n\geq1} \int_0^\infty \IE\bklr{h(\Zx(1))\big| L_1=n}\I[L_t=n] dt \\
	& =\IE \sum_{n\geq1} \IE\bklr{h(\Zx(1))\big| L_1=n}Y_n 
	 =\sum_{n\geq 1} \IE\bklr{ h(\Zx(1))\big| L_1=n}\IE Y_n,
}
where we have used dominated convergence multiple times to interchange expectation, integration and summation, along with the fact that~$\IE \bklr{h(\Zxt)\big|L_t=n}$ only depends on~$n$ and not on~$t$ and can therefore be replaced by~$\IE \bklr{h(\Zx(1))\big|L_1=n}$ (or with~$t$ being replaced by any other fixed positive time).
This leads to
\be{
\abs{f(\tx)}\leq \frac{1}{2}\norm{ h}_\infty \sum_{n\geq 1} \IE Y_n
= \norm{ h}_\infty \sum_{n\geq 1}\frac{1}{n(n-1+s)}\leq \frac{(s+1)}{s}\norm{ h}_\infty,
}
which is \eq{18}.

\smallskip
\noindent{\bf Preliminaries for partial derivatives.}
To show the existence and bounds for the partial derivatives, 
 we need some couplings. Let $\te_i$ denote the unit vector with a one in the $i$th coordinate (and zeros in all others with dimension from context).
Fix~$m\geq 0$ and~$1\leq i_1,\dots,i_{m+1}\leq K-1$. Let~$\tx = (x_1, x_2, \ldots, x_{K-1})\in\Delta_K$. Choose~$\eps_1,\dots,\eps_{m+1}>0$ arbitrarily, but small enough that~$x_K:=1-\sum_{j=1}^{K-1} x_j>\sum_{j=1}^{m+1}\eps_j$, or equivalently, that~$\tx+\sum_{j=1}^{m+1}\eps_j\te_{i_j}\in\Delta_K$. Then, proceed with the following steps. 
\begin{enumerate}[$(i)$]
\item Let~$L_t$ be the pure death process as described above. 
\item Given~$L_t$, let~$\tM :=(\tB, \tN)\sim \MN_{m+1+K}(L_t;\e_1, \dots, \e_{m+1}, \tx , x_K-\sum_{j=1}^{m+1}\eps_j)$, where~$\tB=(B_1,\dots,B_{m+1})$ and~$\tN=(N_1,\dots,N_K)$.
\item Given~$L_t$ and~$\tM$, let
\ben{\label{21}
( \d_1, \dots, \d_{m+1}, \tD_\tx) \sim \Dir(\tB, \ta+\tN).
}
\item  Set~$\teps_j = \e_j \te_{i_j}$ for~$1\leq j\leq m+1$. 
As described immediately below, basic facts 
about the multinomial and Dirichlet distributions imply that
\ben{\label{22}
\tD_\tx \eqlaw \Zxt
}
and that
\ben{\label{23}
\tD_{\tx} + \sum_{j\in A} \d_{j}\te_{i_j} \eqlaw \tZ_{\tx+\sum_{j\in A} \teps_{j}}(t)
}
for any subset~$A\subset\{1,\dots,m+1\}$.
\end{enumerate}
To see why~\eq{22} and~\eq{23} are true, use the
following standard facts.
\begin{itemize}
\item Let~$(\xi_1, \dots, \xi_{p-1}) \sim \Dir(y_1,\ldots,y_p)$. If~$A=\{i_1,\ldots,i_j\}\subset\{1,\ldots, p-1\}$
is any subset of indices, then
\be{
\bklr{\xi_{i_1},\ldots, \xi_{i_j}  } \sim \Dir\bbklr{y_{i_1},\ldots,y_{i_j}, y_p+\sum_{k\not\in A} y_k}.
}
Furthermore, letting~$\boldsymbol{\xi}^{(k)}\in \IR^{p-2}$ denote~$(\xi_1, \dots, \xi_{p-1})$ with the~$k$th coordinate
removed, 
we have for~$i<k$ (and a similar statement for $i>k$),
\be{
\boldsymbol{\xi}^{(k)}+\te_i \xi_k
\sim \Dir\bklr{y_{1},\ldots,y_{i-1}, y_{i}+y_{k},y_{i+1} \ldots, y_{k-1}, y_{k+1}, \ldots, y_{p} }.
} 
\item Let~$(\zeta_1,\ldots,\zeta_p)\sim \MN_p(b; y_1,\ldots,y_p)$. If~$A=\{i_1,\ldots,i_j\}\subset\{1,\ldots, p\}$
is any 
subset of indices, then
\be{
\bbklr{\zeta_{i_1},\ldots, \zeta_{i_j}, \sum_{k\not\in A} \zeta_k} \sim \MN_{j+1}\bbklr{b; y_{i_1},\ldots,y_{i_j}, \sum_{k\not\in A} y_k}.
}
Furthermore, letting~$\boldsymbol{\zeta}^{(k)}\in \IR^{p-1}$ denote~$(\zeta_1, \dots, \zeta_{p})$ with the~$k$th coordinate
removed, 
we have for~$i<k$ (and a similar statement for $i>k$),
\be{
\boldsymbol{\zeta}^{(k)}+\te_i \zeta_k
\sim \MN_{p-1}\bklr{b; y_{1},\ldots,y_{i-1}, y_{i}+y_{k},y_{i+1} \ldots, y_{k-1}, y_{k+1}, \ldots, y_{p} }.
} 
\end{itemize}
The first item above follows from the usual decomposition of the components of the Dirichlet distribution in terms of ratios of gamma variables, and the second is straightforward from the 
probabilistic description of the multinomial distribution.
%
%

\smallskip
\noindent {\bf Existence of partial derivatives and bounds (\ref{19}).} To ease notation, for vectors $\tx,\ty$ and a function $g$, define $\Delta_{\ty}g(\tx)=g(\tx+\ty)-g(\ty)$ (context should clarify when we mean the simplex or the difference operator). Assume now that $h\in\BC^{m,1}(\-\Delta_K)$ for some $m\geq 0$, let~$1\leq i_1,\dots,i_{m+1}\leq K-1$, and recall the coupling and associated notation defined above. For any~$1\leq k\leq m+1$, we have
\bes{
	\frac{\babs{\Delta_{\teps_{1}}\cdots\Delta_{\teps_{k}}f(\tx)}}{{\eps_1\cdots\eps_k}}
	& = \frac{1}{2{\eps_1\cdots\eps_k}}\bbbabs{\int_{0}^\infty \IE \bklr{\Delta_{\teps_{i}}\cdots\Delta_{\teps_{k}}\bklr{{ h}(\tZ_{\cdot}(t))}(\tx)}dt} \\
	& = \frac{1}{2{\eps_1\cdots\eps_k}}\bbbabs{\int_{0}^\infty \IE \bklr{\Delta_{\teps_{i}}\cdots\Delta_{\teps_{k}}{{ h}(\tD_{\tx})}}dt} \\
	& \leq \frac{\abs{h}_{k-1,1}}{2{\eps_1\cdots\eps_k}}\int_{0}^\infty {\IE \bklr{\delta_1\cdots\delta_m}}dt,
}
where in the last step we have applied Lemma~\ref{lem2}. Now using formulas for Dirichlet and multinomial moments, we have
\be{
	\IE \bklr{\delta_1\cdots\delta_m\bmid L_t,M} = \frac{B_1\cdots B_m}{(L_t+s)(L_t+s+1)\cdots(L_t+s+m-1)}
}
and
\be{
	\IE \bklr{B_1\cdots B_m\bmid L_t} = \eps_1\cdots\eps_k L_t(L_t-1)\cdots(L_t-m+1).
}
Thus, 
\bes{
	\frac{1}{\eps_1\cdots\eps_k}\int_0^\infty  \IE\bklr{\delta_1\cdots\delta_k} dt 
	& = \sum_{n\geq 1}\frac{n(n-1)\cdots(n-k+1)}{(n+s)(n+s+1)\cdots(n+s+k-1)}\IE Y_n \\
	& = \sum_{n\geq 1}\frac{2(n-1)\cdots(n-k+1)}{(n+s-1)(n+s)\cdots(n+s+k-1)} = \frac{2}{k(s+k-1)}.
}
Hence, it follows that
\be{\label{24}
\frac{\babs{\Delta_{\teps_{1}}\cdots\Delta_{\teps_{k}}f(\tx)}}{{\eps_1\cdots\eps_k}} \leq \frac{\abs{h}_{k-1,1}}{k(s+k-1)} =: M_k
}
Since~$\tx$ and~$\eps_1,\cdots,\eps_k$ are arbitrary, 
\eq{31} in Lemma~\ref{lem5} is satisfied, and we conclude that~$f\in\BC^{m,1}(\-\Delta_K)$ and that
for $k=1,\ldots, m+1$,
$\abs{f}_{k-1,1} \leq \frac{\abs{h}_{k-1,1}}{k(s+k-1)}$, which is \eq{19}.

\smallskip
\noindent {\bf Extension to~$\boldsymbol{\-\Delta_K}$.} Assume now~$m\geq 2$. Since~$f\in\BC^{m,1}(\-\Delta_K)$, we have in particular that the~$f_i$ and the~$f_{ij}$ can be extended continuously and uniquely to the boundary of~$\-\Delta_K$. Since the left hand side of \eq{16} only consists of finite sums and continuous transformations of the~$f_i$ and~$f_{ij}$ and is equal to the right hand side of \eq{16} on~$\Delta_K$, it follows
that \eq{16} also holds on the boundary of~$\-\Delta_K$.
\end{proof}

\subsection{Proof of Theorem~\ref{THM3}}

\begin{proof}[Proof of Theorem~\ref{THM3}] Since $h\in\BC^{2,1}(\-\Delta_K)$, Theorem~\ref{THM5} implies that there is  a function $f\in\BC^{2,1}(\-\Delta_K)$ solving \eq{16}. Exchangeability  implies
\bes{
0&=\tsfrac{1}{2}\IE[ (\tW'-\tW)^t \Lambda^{-t} (\nabla f(\tW')+\nabla f(\tW))] \\
	&=\IE[ (\tW'-\tW)^t \Lambda^{-t}\nabla f(\tW))] +\tsfrac{1}{2}\IE[ (\tW'-\tW) ^t\Lambda^{-t} (\nabla f(\tW')-\nabla f(\tW))],
}
and applying the linearity condition~\eq{7} yields
\bes{
&\IE [(\ta-s\tW)^t \nabla f(\tW)] \\
&\qquad= -\tsfrac{1}{2}\IE[ (\tW'-\tW)^t \Lambda^{-t} (\nabla f(\tW')-\nabla f(\tW))]-\IE [\tR^t\Lambda^{-t}\nabla f(\tW)]. 
}
By the fundamental theorem of calculus,
\bes{
	f_i(\tw') & = f_i(\tw) + \int_{0}^1\sum_{j=1}^{K-1}(w'_j-w_j)f_{ij}(\tw+(\tw'-\tw)t)dt \\
	& = f_i(\tw) + \sum_{j=1}^{K-1}(w'_j-w_j)f_{ij}(\tw) + \sum_{j=1}^{K-1}\int_{0}^1(w'_j-w_j)\bklr{f_{ij}(\tw+(\tw'-\tw)t)-f_{ij}(\tw)}dt.
}
Since~$f_{ij}$ is Lipschitz continuous, 
\be{
	\babs{f_{ij}(\tw+(\tw'-\tw)t)-f_{ij}(\tw)}
	\leq \abs{f}_{2,1}t\sum_{k=1}^{K-1}\abs{w'_k-w_k};
}	
hence, there are~$\tilde{Q}_{ijk} = \tilde{Q}_{ijk}(\tw,\tw',f)$ such that~$\abs{\tilde{Q}_{ijk}}\leq \abs{f}_{2,1}$ and 
\bes{
&(\tw'-\tw)^t \Lambda^{-t} (\nabla f(\tw')-\nabla f(\tw)) \\
	&\qquad= \sum_{m,i,j} (\Lambda^{-1})_{i,m} (w_m'-w_m)(w_j'-w_j) f_{i j}(\tw) \\
		&\qquad\qquad+\frac{1}{2}\sum_{m,i,j,k} (\Lambda^{-1})_{i,m} (w_m'-w_m)(w_j'-w_j)(w_k'-w_k)\tilde{Q}_{ijk}.
}
Combining the previous three displays, we have
\ban{
&\IE\left[ \sum_{i,j=1}^{K-1} W_i(\delta_{i j}-W_j) f_{i j}(\tW)+\sum_{i=1}^{K-1}(a_i-s W_i) f_i(\tW)\right] \notag \\
&\quad=\IE\left[\sum_{i,j=1}^{K-1} \left(W_i(\delta_{i j}-W_j) -\frac{1}{2}\sum_{m=1}^{K-1}(\Lambda^{-1})_{i,m} (W_m'-W_m)(W_j'-W_j) \right)f_{i j}(\tW)\right]\notag  \\
&\qquad\quad-\frac{1}{2}\sum_{m,i,j,k}(\Lambda^{-1})_{i,m} \IE\left[(W_m'-W_m)(W_j'-W_j)(W_k'-W_k)\tilde{Q}_{ijk}\right] \notag \\
&\qquad\quad\qquad -\IE \left[\sum_{i,j} R_j(\Lambda^{-1})_{i,j}f_i(\tW)\right].\notag
}
We can further simplify the first summand above to
\bes{
\IE\left[\sum_{i,j,m} (\Lambda^{-1})_{i,m} \left( \Lambda_{m,i} W_i(\delta_{i j}-W_j) -\frac{1}{2} (W_m'-W_m)(W_j'-W_j) \right)f_{i j}(\tW)\right],
}
and now the theorem follows from judicious use of the triangle inequality and the bound \eq{19} from Theorem~\ref{THM5}.

 If~$\Lambda$ is a multiple of the identity matrix
 then, following ideas of \cite{Rollin2008}, 
  the proof is nearly identical but started from
\bes{
f(\tW)-f(\tW')&=\sum_{i=1}^{K-1} (W_i'-W_i) f_i(\tW)+\frac{1}{2}\sum_{i,j=1}^{K-1} (W_i'-W_i)(W_j'-W_j) f_{i j}(\tW) \\
	&\qquad+\frac{1}{6} \sum_{i,j,k=1}^{K-1}(W_i'-W_i)(W_j'-W_j) (W_k'-W_k) \tilde S_{ijk},
}
where~$\tilde S_{ijk}=\tilde S_{ijk}(\tW,\tW',f)$ satisfies $\abs{\tilde S_{ijk}}\leq \abs{f}_{2,1}$. 
From here, the proof follows as above by taking expectation, noting that~$\IE[f(\tW)-f(\tW')]=0$ (since~$\law(\tW)=\law(\tW')$)
and that the expectation of the first term on the right hand side above can be simplified using the linearity condition~\eq{7}.

The bound on the convex set distance \eq{9} directly follows from \eq{8} and Lemma~\ref{lem9}.
\end{proof}

\subsection{Auxiliary results}

In what follows, we define, as usual,~$\Delta_\ty  g(\tx)=g(\tx +\ty)-g(\tx)$ and denote by~$\te_i$ the~$i$th unit vector in~$\IR^n$ .

\begin{lemma} \label{lem2}
Let~$U\subset\IR^n$ be a convex open set, and let~$g\in\BC^{m,1}(U)$ for some~$m\geq 0$. Let~$\tx\in U$, let~$1\leq k\leq m+1$, and let~$\ty\s 1, \ldots, \ty\s k\in\IR^n$ be such that~$\tx + \sum_{i=1}^j \ty\s{j} \in U$ for all~$1\leq j\leq k$.
Then, if~$k\leq m$, 
\be{
\left|\left(\prod_{i=1}^k \Delta_{\ty\s i}\right) g(\tx)\right|\leq \abs{g}_{k} \prod_{i=1}^k \norm{\ty\s i}_1,
}
and if~$k=m+1$, the same estimate holds with~$\abs{g}_{k}$ replaced by~$\abs{g}_{m,1}$ on the right hand side.
\end{lemma}
\begin{proof} Assume~$k\leq m$.
Applying the easy identity
\be{
g(\tx +\ty)-g(\tx)=\int_0^1 \sum_{i=1}^n \frac{\partial g}{\partial x_i} (\tx + t\ty) \ty_i dt
}
repeatedly~$k$ times yields
\ben{\label{25}
\left(\prod_{i=1}^k \Delta_{\ty\s i}\right) g(\tx)=\int_{[0,1]^k} \sum_{i_1, \ldots, i_k=1}^{n}\frac{\partial^k g}{\prod_{ j =1}^k \partial x_{i_j }} \left(\tx + \sum_{j=1}^k \ty\s{j} t_j\right)\prod_{j=1}^k y\s{j}_{i_j} d\textbf{t}.
}
Thus
\be{
\left|\left(\prod_{i=1}^k \Delta_{\ty\s i} \right) g(\tx)\right| \leq \abs{g}_k \sum_{i_1, \ldots, i_k=1}^{n} \prod_{j=1}^k \abs{ y\s{j}_{i_j}} =  \abs{g}_k \prod_{i=1}^k \norm{\ty\s i}_1. 
}
For~$k=m+1$, use~\eq{25} for $k=m$ to find
\bes{
\left(\prod_{i=1}^{m+1} \Delta_{\ty\s i}\right) g(\tx)&=\int_{[0,1]^m} \sum_{i_1, \ldots, i_m=1}^{n}\frac{\partial^m \Delta_{\ty\s{m+1}}g}{\prod_{ j =1}^m \partial x_{i_j }} \left(\tx + \sum_{j=1}^m \ty\s{j} t_j\right)\prod_{j=1}^k y\s{j}_{i_j} d\textbf{t} \\
&=\int_{[0,1]^m} \sum_{i_1, \ldots, i_m=1}^{n}\Delta_{\ty\s{m+1}}\frac{\partial^m g}{\prod_{ j =1}^m \partial x_{i_j }} \left(\tx + \sum_{j=1}^m \ty\s{j} t_j\right)\prod_{j=1}^k y\s{j}_{i_j} d\textbf{t}.
}
Since the $m+1$ partials are Lipschitz, we find
\be{
\bbbbabs{\Delta_{\ty\s{m+1}}\frac{\partial^m g}{\prod_{ j =1}^m \partial x_{i_j }}\left(\tx + \sum_{j=1}^m \ty\s{j} t_j\right)}
	\leq \abs{g}_{m,1} \norm{\ty\s{m+1}}_1,
}
and the result now easily follows by combining this with the previous display.
\end{proof}

\begin{lemma}\label{lem3} Let~$U\subset\IR^n$ be an open convex set, and let~$g:U\to\IR$ be a function. Then,~$g$ is~$M$-Lipschitz continuous with respect to the~$L_1$-norm, if and only if it is 
coordinate-wise~$M$-Lipschitz continuous; that is,  
\be{
	\sup_{x\in U}\sup_u \frac{\abs{\Delta_u g(x)}}{\abs{u}}\leq M.
}
\end{lemma}
\begin{proof} It is clear that if~$g$ is~$M$-Lipschitz continuous, then it is in particular~$M$-Lipschitz continuous in each coordinate. The reverse direction is easily proved using convexity and a telescoping sum argument along the coordinates. 
\end{proof}

We were not able to locate the next two lemmas in the literature; hence, we give self-contained proofs.
There is strong resemblance with the theory of \emph{bounded~$k$-th variation}, see for example \cite[Theorem~11]{Russell1973}, but we were not able to find a result that would directly apply to our situation; we also refer to recent survey textbooks \cite{Mukhopadhyay2012} and \cite{Appell2014}.

In what follows, we assume that~$u$ and~$v$ appearing in terms like~$\Delta_u f(z)$,~$\Delta_u\Delta_v f(z)$,~$\Delta_{u\eps_i} g(\tx)$ and~$\Delta_{u\eps_i}\Delta_{v\eps_j} g(\tx)$ are such that~$z+u$,~$z+u+v$,~$\tx+u\te_i$ and~$\tx+u\te_i+v\te_j$ are within the domains of the functions being evaluated.

\begin{lemma}\label{lem4} Let~$f:(a,b)\to\IR$ be a function. If
\ben{\label{26}
	M_1:=\sup_{z}\sup_{u}\frac{\abs{\Delta_uf(z)}}{\abs{u}}<\infty,		\qquad
	M_2:=\sup_{z}\sup_{  u,v}\frac{\abs{\Delta_u\Delta_v f(z)}}{\abs{uv}}<\infty,		
}	
then~$f$ is differentiable and~$f'$ is~$M_2$-Lipschitz-continuous. 
\end{lemma}
\begin{proof}
Since, by the first condition of~\eq{26}, $f$ is $M_1$-Lipschitz, Rademacher's theorem implies that there is a dense set $E\subset (a,b)$ on which $f$ has a derivative $f'$. On $E$, the second condition of~\eq{26} implies that $f'$ is $M_2$-Lipschitz, and so by Kirszbraun's theorem,~$f'$ can be extended to an $M_2$-Lipschitz function $\tilde f'$ on $(a,b)$. We show that for $x\not\in E$, $\tilde f'(x)$ is in fact the derivative of $f$ at~$x$. Fix $\eps>0$. Let $x'\in E$ such that $\abs{x'-x}<\eps/(3M_2)$, such that $\abs{h^{-1}\Delta_hf(x')-f'(x')}\leq \eps/3$, and such that $\abs{\tilde f'(x')-\tilde f'(x)}< \eps/3$. Then for any $0<h\leq \eps$,
\ba{
 \bbbabs{\frac{\Delta_h f(x)}{h}- \tilde f'(x)}  
	&\leq \bbbabs{\frac{\Delta_h f(x)}{h}-\frac{\Delta_h f(x')}{h}}
	+ \bbbabs{\frac{\Delta f(x')}{h}-\tilde f'(x')} 
	+ \babs{\tilde f'(x')-\tilde f'(x)} \\
	&= \bbbabs{\frac{\Delta_{x-x'}\Delta_h f(x')}{h}}
	+ \bbbabs{\frac{\Delta f(x')}{h}- f'(x')} 
	+ \babs{\tilde f'(x')-\tilde f'(x)} \\
	&\leq \frac{\eps}{3} + \frac{\eps}{3}+ \frac{\eps}{3}=\eps.
}
Hence, $\lim_{h\to0}h^{-1}\Delta_h f(x) = \tilde f'(x)$, as desired.
\end{proof}

\begin{lemma}\label{lem5} Let~$U\subset\IR^n$ be an open convex set, let~$g:U\to\IR$ be a bounded function, and let~$m\geq 0$. If, for each~$1\leq k\leq m+1$, there is a constant~$M_k<\infty$, such that, for each set of indices~$1\leq i_1,\dots,i_k\leq n$,
\ben{\label{31}
	\sup_{x\in U }\sup_{u_1,\dots,u_k}\frac{\abs{\Delta_{u_1\te_{i_1}}\cdots\Delta_{u_k\te_{i_k}} g(x)}}{\abs{u_1\cdots u_k}} \leq M_k,		
}	
then~$g\in \BC^{m,1}(U)$
 and 
\ben{\label{32}
	\abs{g}_{k,1}\leq M_{k+1},\qquad 0\leq k\leq m.
}
\end{lemma}
\begin{proof} If~$m=0$, the result is immediate since \eq{31} is just the coordinate-wise~$M_1$-Lipschitz condition, which implies that~$g$ is Lipschitz, and \eq{32} follows from Lemma~\ref{lem3}. Now, assume~$m\geq 1$. Fix a set of~$m$ indices~$1\leq i_1,\dots,i_m\leq n$. We proceed by induction and start with~$k=1$. Let~$\tx= (x_1,\dots,x_n)\in U~$, let~$a<x_i<b$ such that 
\be{
	t(x_1,\dots,x_{i_1-1},a,x_{i_1-1},\dots,x_n) 
		+ (1-t)(x_1,\dots,x_{i_1-1},b,x_{i_1-1},\dots,x_n) \in U, \qquad 0\leq t\leq1,
}
and, with~$\tx_z = (x_1,\dots,x_{i_1-1},z,x_{i_1-1},\dots,x_n)$, let~$f(z) = g(\tx_z)$. By the assumptions on~$g$, we have
\bg{
	\sup_{z}\sup_{u}\frac{\abs{\Delta_{u}f(z)}}{\abs{u}} 
	= \sup_{z}\sup_{u_1}\frac{\abs{\Delta_{u_1\e_{i_1}}g(\tx_z)}}{\abs{u_1}}\leq M_1<\infty,\\
	\sup_{z}\sup_{u,v}\frac{\abs{\Delta_{u}\Delta_{v}f(z)}}{\abs{uv}}
	=\sup_{z}\sup_{u_1,u_2}\frac{\abs{\Delta_{u_1\e_{i_1}}\Delta_{u_2\e_{i_1}}g(\tx_z)}}{\abs{u_1u_2}}\leq M_2<\infty
} 
(note that in the second expression, the second difference is also in the direction~$\e_{i_1}$)
so that the conditions \eq{26} are satisfied. Applying Lemma~\ref{lem4}, we conclude that~$f'(z)=\frac{\partial}{\partial x_{i_1}} g(\tx_z)$ exists and that it is~$M_2$-Lipschitz continuous in direction~$i_1$, but the same argument there together with \eq{31} yields~$M_2$-Lipschitz continuity in any other direction, so that by Lemma~\ref{lem3},~$\frac{\partial}{\partial x_{i_1}} g(\tx)$ is~$M_2$-Lipschitz. Since~$\tx$ was arbitrary,~$\frac{\partial}{\partial x_{i_1}} g(\tx)$ exists in all of~$U$ and is~$M_2$-Lipschitz, which concludes the base case. 

Assume now that~$1<k<m$ and that~$\frac{\partial^{k-1}}{\partial x_{i_1}\cdots\partial x_{i_{k-1}}} g(\tx)$ exists in all of~$U$. Let~$x\in U$, let~$a$,~$b$ and~$\tx_z$ be as before, and let~$f(z) = \frac{\partial^{k-1}}{\partial x_{i_1}\cdots\partial x_{i_{k-1}}} g(\tx_z)$.
From the assumptions on~$g$ and since~$\frac{\partial^{k-1}}{\partial x_{i_1}\cdots\partial x_{i_{k-1}}} g(\tx)$ exists, we have
\bg{
	\sup_{z}\sup_{u}\frac{\abs{\Delta_{u}f(z)}}{\abs{u}} 
	= \sup_{z}\sup_{u_{k}}\bbbabs{\lim_{u_1\to0}\cdots\lim_{u_{k-1}\to0}\frac{\Delta_{u_1\te_{i_1}}\cdots\Delta_{u_{k}\te_{i_{k}}}g(\tx_z)}{u_1\dots u_{k}}} \leq M_{k} < \infty \\
	\sup_{z}\sup_{u,v}\frac{\abs{\Delta_{u}\Delta_v f(z)}}{\abs{uv}} 
	= \sup_{z}\sup_{u_{k},u_{k+1}}\bbbabs{\lim_{u_1\to0}\cdots\lim_{u_{k-1}\to0}\frac{\Delta_{u_1\te_{i_1}}\cdots\Delta_{u_{k}\te_{i_{k}}}\Delta_{u_{k+1}\te_{i_{k}}}g(\tx_z)}{u_1\dots u_{k}u_{k+1}}} \leq M_{k+1} < \infty 
}
so that the conditions \eq{26} are satisfied. Applying Lemma~\ref{lem4}, we conclude that~$f'(z)=\frac{\partial^{k}}{\partial x_{i_1}\cdots\partial x_{i_{k}}} g(\tx_z)$ exists
and that it is~$M_{k+1}$-Lipschitz continuous in direction~$i_k$, but the same argument there together with \eq{31} yields~$M_{k+1}$-Lipschitz continuity in any other direction, so that by Lemma~\ref{lem3},~$\frac{\partial^{k}}{\partial x_{i_1}\cdots\partial x_{i_{k}}} g(\tx_z)$ is~$M_{k+1}$-Lipschitz. Since~$\tx$ was arbitrary,~$\frac{\partial^{k}}{\partial x_{i_1}\cdots\partial x_{i_{k}}} g(\tx_z)$ exists in all of~$U$ and is~$M_{k+1}$-Lipschitz, which concludes the induction step. 
\end{proof}

The following is a specialisation of \cite[Lemma~2.1]{Bentkus2003} to the convex set metric; we need some notation first. Let~$A\subset \IR^K$ be convex, let~$d(\tx,A) = \inf_{\ty\in A}|\tx-\ty|$, and define the sets
\ben{\label{33}
	A^\eps = \{\tx\in \IR^K\,:\,d(\tx,A)\leq \eps\},\qquad A^{-\eps} = \{\tx\in A\,:\, B(\tx;\eps)\subset A\},
}
where~$B(\tx;\eps)$ is the closed ball of radius~$\eps$ around~$\tx$.
\def\cA{\mathcal{C}_{K-1}}

\begin{lemma}[{\cite[Lemma~2.1]{Bentkus2003}}] \label{lem6} Let~$\cC_K$ be the family of convex sets of\/~$\IR^K$, and for fixed~$\eps>0$, let~$\{\phi_{\eps,A}; \, A\in\cC_K\}$ be a family of functions satisfying
\ben{\label{34}
	0\leq \phi_{\eps,A} \leq 1,
	\qquad 
	\text{$\phi_{\eps,A}(\tx) = 1$ for~$\,\tx\in A$},
	\qquad
	\text{$\phi_{\eps,A}(\tx) = 0$ for~$\,\tx\not\in A^{\eps}$.}
}
Then, for any two random vectors~$\tX$ and~$\tY$, 
\bes{
	&\sup_{A\in\cC_K}\babs{\IP[\tX\in A]-\IP[\tY\in A]} \\
	&\quad \leq \sup_{A\in\cC_K}\babs{\IE\phi_{\eps,A}(\tX)-\IE\phi_{\eps,A}(\tY)}
	+  \sup_{A\in\cC_K}\max\bklg{\IP[\tY\in A\setminus A^{-\eps}],\IP[\tY\in A^\eps\setminus A]}
}
\end{lemma}

\def\cS{\mathcal{S}}

\begin{lemma}[Smoothing operator] \label{lem7}Let~$f:\IR^n\to \IR$ be a bounded and Lebesgue measurable function. For~$\eps>0$, define the smoothing operator~$\cS_\eps$ as
\be{
	(\cS_\eps f)(\tx) = \frac{1}{(2\eps)^n}\int\limits_{x_1-\eps}^{x_1+\eps}\cdots\int\limits_{x_n-\eps}^{x_n+\eps}f(z)\,dz_n\cdots dz_1.
}
Then, for any $m\geq 1$, we have that $\cS_\eps^m f\in\BC^{m-1,1}(\IR^n)$, and for fixed~$\tx\in\IR^n$,~$(\cS_\eps^ m f)(\tx)$ does not depend on~$f(\ty)$,~$\ty\in\IR^n\setminus B(\tx;mn^{1/2}\eps)$. Moreover, we have the bounds
\ben{\label{35}
	\norm{\cS^m_\eps f}_{\infty}\leq\norm{f}_\infty,
	\qquad
	\abs{\cS^m_\eps f}_{k-1,1}\leq \frac{\norm{f}_\infty}{\eps^k}, \quad 1\leq k\leq m.
}
\end{lemma}

\begin{proof} The claim that $f(\tx)$ does not depend on $f(\ty)$,~$\ty\in\IR^n\setminus B(\tx;mn^{1/2}\eps)$, is a straightforward consequence of the definition, as is the bound 
\ben{\label{36}
	\norm{\cS_\eps f}_\infty \leq \norm{f}_\infty.
}
Now, it is easy to see that for $u>0$ and $1\leq i\leq n$,
\bes{
	\abs{\Delta_{u\te_{i}}\cS_\eps f (\tx)}
	\leq \begin{cases}
		\displaystyle2\norm{f}_\infty& \text{if $u>2\eps$,}\\[2ex]
		\displaystyle\frac{u\norm{f}_\infty}{\eps}& \text{if $u\leq 2\eps$,}\\
	\end{cases}
}
so that $\abs{\Delta_{u\te_{i}}\cS_\eps f (\tx)}\leq u\norm{f}_{\infty}/\eps$ for all $x$ and all $u$, which implies that
\ben{\label{37}
\bbbnorm{\frac{\Delta_{u\te_{i}}\cS_\eps f}{u}}_\infty
	\leq \frac{\norm{f}_\infty}{\eps}.
}
Fix $1\leq k\leq m$, $u_1,\dots,u_k>0$ and $1\leq i_1,\dots,i_k\leq n$. Noting that~$\Delta_{u\te_i}\cS_\eps g = \cS_\eps \Delta_{u\te_i}g$, we can write
\be{
	\Delta_{u_1\te_{i_1}}\cdots\Delta_{u_k\te_{i_k}}\cS_\eps^m = (\Delta_{u_1\te_{i_1}}\cS_\eps)\cdots(\Delta_{u_k\te_{i_k}}\cS_\eps) \cS_\eps^{l-m}.
}
Applying \eq{37} repeatedly~$k$ times and if $k<m$ applying in addition \eq{36}, we obtain \eq{31} with $M_k = \norm{f}/\eps^k$, so that the claim follows from Lemma~\ref{lem5}.
\end{proof}

\begin{lemma}\label{lem8} Let~$\eps>0$, and let~$A\subset \IR^n$ be convex. There exists a function $\phi=\phi_{\eps,A}\in\BC^{2,1}(\IR^n)$ satisfying~\eq{34} with 
\ben{\label{38}
	\abs{\phi}_1\leq \frac{9n^{1/2}}{\eps},
	\qquad
	\abs{\phi}_2\leq \frac{81n}{\eps^2},
	\qquad	
	\abs{\phi}_{2,1}\leq \frac{729n^{3/2}}{\eps^3}.
}
\end{lemma}
\begin{proof} Let~$\delta = \frac{\eps}{9\sqrt{n}}$. Define
\be{
	\phi(\tx) = \cS_\delta^3 I_{A^{\eps/3}}(\tx);
}
the claim then follows from Lemma~\ref{lem7}.
\end{proof}

\begin{lemma}\label{lem9} Let~$\cC_{K-1}$ be the class of convex sets on~$\IR^{K-1}$. Let~$\tZ\sim \Dir(a_1,\dots,a_K)$ and assume 
\ben{\label{39}
	\abs{\IE h(\tW) - \IE h(\tZ)} \leq c_0\abs{h}_0 + c_1\abs{h}_1 + c_2\abs{h}_2 + c_3\abs{h}_{2,1},
}
for any $h\in\BC^{2,1}(\-\Delta_K)$. 
Then there is a constant~$C>0$ depending only on~$a_1,\dots,a_K$ such that
\ben{\label{40}
	\sup_{A\in \cC_{K-1}}\abs{\IP[\tW\in A]-\IP[\tZ\in A]} \leq c_0+C(c_1+c_2+c_3)^{\theta/(3+\theta)},
}
where 
\ben{
	\theta = \frac{\theta_\wedge}{\theta_\wedge+\theta_\circ},\qquad \theta_\wedge = 1\wedge\min\{a_1,\dots,a_K\},\qquad \theta_\circ=\sum_{i=1}^K\bklr{1-1\wedge a_i}.
}
\end{lemma}
\begin{proof} Since both $\tW$ and $\tZ$ take values in $\-\Delta_K$, we may assume without loss of generality that~$A\subset  \-\Delta_K$. 
Fix~$\eps>0$; from Lemma~\ref{lem6} we have
\bes{
	& \sup_{A\in \cC_{K-1}}\abs{\IP[\tW\in A]-\IP[\tZ\in A]} \\
	& \qquad\leq \sup_{A\in \cC_{K-1}}\abs{\IE \phi_{\eps,A}(\tW) - \IE \phi_{\eps,A}(\tZ)} + \sup_{A\in \cC_{K-1}} \IP[\tZ\in A^\eps\setminus A]\vee \IP[\tZ\in A\setminus A^{-\eps}] \\
	& =: R_1 + R_2,
}
where the $\phi_{\eps,A}$ are chosen as in Lemma~\ref{lem8}. Using \eq{39} and \eq{38} 
\be{
	R_1\leq c_0 + \frac{9(K-1)^{1/2}c_1}{\eps} + \frac{81(K-1)c_2}{\eps^2} + \frac{729(K-1)^{3/2}c_3}{\eps^3}.
}
In order to bound~$R_2$ we proceed as follows. Let~$\delta\geq\eps$ (to be chosen later), and consider~$\-\Delta^{-\delta}_{K-1}$, the~$\delta$-shrinkage of~$\-\Delta_{K-1}$ as defined in~\eq{33}. 
For given convex~$A\subset \-\Delta_{K-1}$, let~$A_\circ = A \cap \-\Delta^{-\delta}_{K-1}$ (which is again convex) and note that
\ba{
	\IP[\tZ\in A^\eps\setminus A] & \leq \IP[\tZ\in \-\Delta_{K-1}\setminus\-\Delta_{K-1}^{-\delta}] + \IP[\tZ\in A^{\eps}_\circ\setminus A_\circ], \\
	\IP[\tZ\in A\setminus A^{-\eps}] & \leq \IP[\tZ\in \-\Delta_{K-1}\setminus\-\Delta_{K-1}^{-\delta}] + \IP[\tZ\in A_\circ\setminus A^{-\eps}_\circ],
}
so that
\be{
	\IP[\tZ\in A^\eps\setminus A]\vee \IP[\tZ\in A\setminus A^{-\eps}]\leq
	\IP[\tZ\in \-\Delta_{K-1}\setminus\-\Delta_{K-1}^{-\delta}] + \IP[\tZ\in A_\circ^\eps\setminus A_\circ]\vee \IP[\tZ\in A_\circ\setminus A_\circ^{-\eps}].
}
Using a union bound and the fact that the marginals of~$\tZ$ have beta distributions,
\bes{
	\IP\bkle{\tZ\in \-\Delta_{K-1}\setminus\-\Delta^{-\delta}_{K-1}} 
		&\leq \sum_{i=1}^K\bklr{\IP[Z_i\leq \delta] + \IP[Z_i\geq 1-\delta]} \\
		& \leq \sum_{i=1}^K\frac{\Gamma(s)}{\Gamma(a_i)\Gamma(s-a_i)}\bbklr{\frac{\delta^{a_i}}{a_i}+\frac{\delta^{s-a_i}}{s-a_i}}
		\leq C\delta^{\theta_\wedge}.
} 
Now, the density~$\psi_\ta$ of the Dirichlet distribution (see~\eq{1}), restricted to~$\-\Delta^{-\delta}_{K-1}$, is bounded by
\ben{\label{41}
	\bbnorm{\psi_\ta\bmid_{\-\Delta^{-\delta}_{K-1}}} \leq \frac{\Gamma(s)}{\prod_{i=1}^K \Gamma(a_i)} \delta^{-\theta_\circ}.
}
From Steiner's formula for convex bodies, which describes the volume of~$\eps$-enlargements of convex bodies (see e.g.\ \cite[Theorem~46]{Morvan2008}), the Hausdorff-continuity of the corresponding coefficients in Steiner's formula (so called \emph{Quermassintegrale}; see \cite[Theorem~50]{Morvan2008}) and compactness of~$\-\Delta_{K-1}$, and the bound \eq{41}, we conclude that there is a constant~$S_K$ that only depends on the dimension~$K$ such that, for convex~$A\subset \-\Delta^{-\delta}_{K-1}$,~$\Vol(A^\eps\setminus A)\leq \eps S_K$ and~$\Vol(A\setminus A^{-\eps})\leq \eps S_K$, so that if~$\delta>\eps$, 
\be{
\IP[\tZ\in A_\circ^\eps\setminus A_\circ]\vee \IP[\tZ\in A_\circ\setminus A_\circ^{-\eps}]\leq \frac{\Gamma(s)}{\prod_{i=1}^K \Gamma(a_i)}(\delta-\eps)^{-\theta_\circ}\cdot \eps S_K
}
(note that an upper bound on~$S_K$ could be obtained in principle by evaluating the coefficients in the Steiner formula for the convex set~$\-\Delta_{K-1}$).
Note that in the previous display if~$\theta_\circ=0$ then the inequality still
holds without the factor of~$(\delta-\eps)$ even if~$\delta=\eps$.
Thus we have
\be{
	\IP[\tZ\in A^\eps\setminus A]\vee \IP[\tZ\in A\setminus A^{-\eps}]\leq C \bkle{\delta^{\theta_\wedge} + \eps\bklr{\I[\theta_\circ>0](\delta-\eps)^{-\theta_\circ}
	+\I[\theta_\circ=0]}}.
}
Choosing~$\delta = \eps^{1/(\theta_\wedge+\theta_\circ)}$, 
we have that~$\delta\geq\eps$,
and~$\delta=\eps$ only if~$\theta_\circ=0$, so that
\be{
	\sup_{A\in \cC_{K-1}}\abs{\IP[\tW\in A]-\IP[\tZ\in A]} \leq C\bbklr{c_0 + \frac{c_1}{\eps} + \frac{c_2}{\eps^2} + \frac{c_3}{\eps^3} + \eps^\theta},
}
for some constant~$C=C(\ta)$.
Without loss of generality we may assume that~$C\geq 1$ in \eq{40} so that if~$c_1+c_2+c_3 \geq 1$, then~\eq{40} is trivially true. If~$c_1+c_2+c_3 < 1$, choose~$\eps = (c_1+c_2+c_3)^{1/\klr{3+\theta}}<1$ and bound both~$1/\eps$ and~$1/\eps^2$ by~$1/\eps^3$; this again yields \eq{40}.
\end{proof}

\section{Proof of Theorem~\ref{THM1}: Wright-Fisher model}\label{sec2}

Recall the description in the introduction of the Wright-Fisher model with neutral mutation in a haploid population of constant size~$N$.
The process is driven by offspring vector having distribution~$\MN(N;1/N, \ldots, 1/N)$, and the 
mutation structure is general with~$K$ types. The process is a time-homogeneous Markov chain~$\tX(0), \tX(1), \ldots$, where~$\tX(n)$ is a~$(K-1)$ dimensional vector that represents the counts of the first~$K-1$ alleles in the population, so~$\tX(n)/N \in \-\Delta_K$.

Since~$(\tX(n))_{n\geq0}$ is a Markov chain on a finite state space, it has a stationary distribution, and we apply Theorem~\ref{THM3} to prove the
bound on the approximation of this stationary distribution by the Dirichlet distribution given by Theorem~\ref{THM1}.
To define a stationary pair~$(\tW, \tW')$, let~$\tX$ be distributed as a stationary distribution of the chain of Theorem~\ref{THM1} 
and let~$\tX'$ be a step in the chain from~$\tX$. 
Set~$\tW=\tX/N$ and~$\tW'=\tX'/N$. 

It is not difficult to see that 
the distribution of~$\tX'$ given~$\tX$ is the first~$K-1$ coordinates of a multinomial with~$N$ trials with success probabilities given by the vector
$\tq(\tX)$, where 
\ben{\label{42}
q_j(\tX)
	=\sum_{k=1}^Kp_{ k  j}\frac{X_ k}{N}  = \sum_{k=1}^{K-1} p_{ k  j}\frac{X_ k}{N} + p_{Kj}\bbbklr{1-\sum_{k=1}^{K-1}\frac{X_k}{N}}. 
}
Hence
\be{
	\IE[W_j'|\tW] 
		 = p_{jj}W_j + \sum_{\substack{k=1\\k\neq j}}^{K-1}p_{ k  j} W_k  + p_{Kj} - \sum_{k=1}^{K-1}p_{Kj} W_k,
}
so that
\bes{
	\IE[W_j'-W_j|\tW] 
		& = -(1-p_{jj})W_j + \sum_{\substack{k=1\\k\neq j}}^{K-1}p_{ k  j} W_k  + p_{Kj} - \sum_{k=1}^{K-1}p_{Kj} W_k \\
		& =\frac{1}{2N}(a_j -sW_j) + R_j(\tW),
}
where
\besn{\label{43}
	R_j(\tW) 
	& =  -\frac{a_j}{2N}+ \bbklr{\frac{s}{2N}-(1-p_{jj})}W_j+\sum_{\substack{k=1\\k\neq j}}^{K-1}p_{ k  j} W_k  + p_{Kj} - \sum_{k=1}^{K-1}p_{Kj} W_k \\
	& =  \bbklr{p_{Kj}-\frac{a_j}{2N}}(1-W_j)+\bbbklr{\,\sum_{\substack{k=1\\k\neq j}}^K\bbklr{\frac{a_k}{2N}-p_{jk}}}W_j+\sum_{\substack{k=1\\k\neq j}}^{K-1}(p_{ k  j}-p_{Kj}) W_k.
}
Thus we are in the setting of Theorem~\ref{THM3} with~$\ta$ as above,~$\Lambda=(2N)^{-1}\times \Id$, and~$\tR$ given by~\eq{43}. 

Applying the theorem is a relatively straightforward but somewhat tedious calculation involving conditioning and multinomial moment formulas. We need the quantities
\bg{
	T_j=p_{K j}+\sum_{ \substack{k=1\\k \neq j}}^{K-1} (p_{ k  j}- p_{K j}) W_{ k },\\
	\sigma_j= p_{Kj} + \sum_{\substack{k=1\\k\neq j}}^K p_{jk},\qquad \tau_j =  p_{Kj} + \sum_{\substack{k=1\\k\neq j}}^K\abs{p_{kj}-p_{Kj}}, \qquad 1\leq j\leq K-1.
}
Note that we can write
\ben{\label{44}
q_j:=q_j(\tX)=W_j(1-\sigma_j)+T_j.
}
We also record the following multinomial moment formula lemma. Let~$(n)_{ k \downarrow}=n(n-1)\cdots(n- k +1)$ denote the falling factorial.
\begin{lemma}\label{lem10}
For~$(\tX, \tX')$ defined above,~$i,j,k\in\{1,\ldots, K-1\}$ all distinct and non-negative integers~$ k _i,  k _j,  k _k$, 
\be{
\IE\left[ \left(X_i'\right)_{ k _i\downarrow}\left(X_j'\right)_{ k _j\downarrow}\left(X_k'\right)_{ k _k\downarrow} \big| \tX \right]
=\left( N \right)_{( k _i+ k _j+ k _k) \downarrow} q_i(\tX)^{ k _i} q_j(\tX)^{ k _j} q_k(\tX)^{ k _k}.
 }
\end{lemma}

 \begin{lemma}\label{lem11}
 For~$(\tW,\tW')$ defined above,
 \bes{
 \IE\left[(W_j'-W_j)^2\big| \tW \right]&=W_j^2\left[-\frac{1}{N}+\frac{2\sigma_j}{N} +\sigma_j^2\left(1-\frac{1}{N}\right)\right] \\
 	&\qquad + W_j\left[\frac{1}{N} - \frac{2T_j}{N}-\frac{\sigma_j}{N}-2T_j\sigma_j\left(1-\frac{1}{N}\right)\right]
	+T_j\left[T_j+\frac{1-T_j}{N}\right].
 }
 \end{lemma}
 \begin{proof}
We first expand
\be{
\IE\left[(W_j'-W_j)^2\big| \tW \right]=\frac{1}{N^2}\IE\left[ X_j'\left(X_j'-1\right) \big| \tW\right]
	-\left(2W_j-\frac{1}{N}\right) \IE[W_j'|\tW] + W_j^2.
}
Using Lemma~\ref{lem10} and the expression for~$\tq$ given at~\eq{44} we find
\be{
\frac{1}{N^2}\IE\left[ X_j'\left(X_j'-1\right) \big| \tW\right]=\frac{N-1}{N} \left(W_j \left(1-\sigma_j\right)+T_j\right)^2,
}
and
\be{
\IE[W_j'|\tW]=W_j \left(1-\sigma_j\right)+T_j.
}
Combining these last three displays and simplifying yields the result.
\end{proof}

\begin{lemma}\label{lem12}
For~$(\tW,\tW')$ defined above, and~$i\not=j$,
\bes{
\IE\left[(W_i'-W_i)(W_j'-W_j) \big| \tW \right]&=W_iW_j\left[-\frac{1}{N}+\frac{\sigma_i+\sigma_j}{N}+\sigma_i\sigma_j\left(1-\frac{1}{N}\right)\right] +T_iT_j\left(1-\frac{1}{N}\right) \\
&\qquad 
 -W_iT_j\left (\sigma_i+\frac{1-\sigma_i}{N}\right) -W_jT_i\left (\sigma_j+\frac{1-\sigma_j}{N}\right).
}
\end{lemma}
\begin{proof}
We first expand
\be{
\IE\left[(W_i'-W_i)(W_j'-W_j)|\tW\right]=\IE [W_i'W_j' | \tW] -W_i\IE[W_j'|\tW]-W_j\IE[W_i'|\tW]+W_i W_j.
}
Using Lemma~\ref{lem10} and the expression for~$\tq$ given at~\eq{44} we find
\be{
\IE [W_i'W_j' | \tW]=\frac{N-1}{N}  \left(W_i \left(1-\sigma_i\right)+T_i\right) \left(W_j \left(1-\sigma_j\right)+T_j\right),
}
and
\be{
W_i\IE[W_j'|\tW]=W_i\left(W_j \left(1-\sigma_j\right)+T_j\right).
}
Combining these last three displays and simplifying yields the result.
\end{proof}

\begin{lemma}\label{lem13}
For~$(\tW,\tW')$ defined above, and~$\lambda=(2N)^{-1}~$,
\ba{
\sum_{i,j=1}^{K-1} \IE &\left| W_i(\delta_{i j}-W_j)-\frac{1}{2 \lambda}\IE[ (W_i'-W_i)(W_j'-W_j)|\tW] \right|\leq N\sum_{i,j=1}^{K-1}(\sigma_i+\tau_i) \left(\sigma_j+\tau_j+\frac{2}{N}\right).
} 
\end{lemma}
 \begin{proof}
 The lemma follows in a straightforward way from Lemmas~\ref{lem11} and~\ref{lem12}, the triangle inequality, that~$0\leq W_i\leq 1$,
 and~$\abs{T_j}\leq \tau_j$.
  \end{proof}

\begin{lemma}\label{lem14}
For~$(\tW, \tW')$ defined above and~$\lambda=(2N)^{-1}~$,
\ba{
\frac{1}{\lambda}\sum_{i,j,k=1}^{K-1}&\IE \abs{(W_i'-W_i)(W_j'-W_j)(W_k'-W_k)}\\
	&\leq \frac{2}{N^{1/2}} \left(\sum_{i=1}^{K-1} \left[\sqrt2+\sqrt N (\tau_i + \sigma_i) \right]\right)^{2} \left(\sum_{i=1}^{K-1} \left[1+\sqrt N (\tau_i + \sigma_i)\right]\right).
}

\end{lemma}
\begin{proof}
Conditional on~$\tX$,~$\tX'$ is distributed as the first~$(K-1)$ entries of a multinomial distribution with~$N$ trials 
and success probabilities given by the vector at~\eq{44}: 
\be{
q_ k :=q_ k (\tX)=\left(W_ k  \left(1-\sigma_ k \right)+T_ k \right).
}
Decompose
\ba{
X_i' - X_i&=X_i' -\IE[X_i'| X_i] + \IE[X_i'| X_i] - X_i \\
	&= [X_i' - (X_i(1-\sigma_i) + NT_i)] + [NT_i - \sigma_iX_i]\\
	&=: E_i + G_i.
}
Using H\"older's inequality followed by Minkowski's inequality, we find
\ban{
\sum_{i,j,k}^{K-1} &\IE \abs{(W_i'-W_i)(W_j'-W_j)(W_k'-W_k)} = \frac{1}{N^3}\sum_{i,j,k}^{K-1}\IE \abs{(E_i + G_i)(E_j + G_j)(E_k + G_k)} \notag \\
	&\leq \frac{1}{N^3}\sum_{i,j,k}^{K-1}\left[\IE (E_i+G_i)^4 \IE (E_j+G_j)^4 \right]^{1/4}\left[\IE (E_k+G_k)^2\right]^{1/2} \notag \\
	&\leq \frac{1}{N^3}\left(\sum_{i=1}^{K-1}\left[(\IE E_i^4)^{1/4}+(\IE G_i^4)^{1/4} \right]\right)^2\sum_{k=1}^{K-1}\left[(\IE E_k^2)^{1/2}+(\IE G_i^2)^{1/2}\right]. \label{45}
}
Now noting that for~$Y\sim\Bin(n,p)$,
\bes{
\IE(Y-np)^4 &= 3(np(1-p))^2 + np(1-p)(1-6p(1-p)) \leq 3(np(1-p))^2 + np(1-p),
}
which, along with the variance formula for the binomial distribution, yields
\ba{
\IE E_i^4 &= 3(X_iq_i(1-q_i))^2 + X_iq_i(1-q_i) \leq 4N^2, \\
\IE E_i^2 &= X_i q_i (1-q_i) \leq N.
}
Plugging these bounds along with~$\abs{G_i} = \abs{NT_i - \sigma_i X_i} \leq N\tau_i + N\sigma_i$ into~\eq{45} yields the result.
\end{proof}

\begin{proof}[Proof of Theorem~\ref{THM1}] We apply Theorem~\ref{THM3} with~$\Lambda=(2N)^{-1}\times \Id$. Using the bounds in Lemmas~\ref{lem13} for~$A_2$ and~\ref{lem14} for~$A_1$ along with a straightforward bound on~$\abs{R_j}$ ($R_j$ given at~\eq{43}) for~$A_1$, we obtain
\ba{
A_1 &\leq  2N\sum_{j=1}^{K-1} \left[ | p_{K j} - \frac{a_j}{2N}| +\sum_{\substack{k=1\\k\not=j}}^{K}\abs{p_{j k}-\frac{a_k}{2N}}+\sum_{\substack{k=1\\k\not=j}}^{K-1} \abs{p_{k j}- p_{Kj}}\right],\\
A_2&\leq N\sum_{i,j=1}^{K-1}(\sigma_i+\tau_i) \bbklr{\sigma_j+\tau_j+\frac{2}{N}}, \\
A_3&\leq \frac{2}{N^{1/2}}\left(\,\sum_{i=1}^{K-1} \bklr{\sqrt{2}+\sqrt{N}(\sigma_i+ \tau_i)}\right)^{2}\left(\,\sum_{i=1}^{K-1} \bklr{1+\sqrt{N}(\sigma_i+ \tau_i)}\right),
}
The final bound in Theorem~\ref{THM1} is now obtained through straightforward manipulations and applying some standard analytic inequalities, in particular 
that~$\abs{x+y}^p\leq 2^{p-1}(\abs{x}^p+\abs{y}^p)$ for~$p\geq1$, and that
\be{
\sum_{i=1}^{K-1}(\sigma_i+\tau_i)\leq 2\sum_{i=1}^{K-1}p_{K i}+
\sum_{i=1}^{K-1}\sum_{\substack{ j\neq i}}^{K}p_{i j}+\sum_{i=1}^{K-1}\sum_{\substack{ j\neq i}}^{K-1}p_{ij}
+(K-2)\sum_{i=1}^{K-1} p_{K i}
\leq K \sum_{i=1}^K\sum_{\substack{j=1\\j\neq i}}^K p_{ij}
=K\mu.
}
\end{proof}

\section{Proof of Theorem~\ref{THM2}: Cannings model}\label{sec3}

Recall the description in the introduction of the Cannings exchangeable
model with neutral PIM mutation in a haploid population of constant size~$N$.
The process is driven by a generic exchangeable offspring vector 
$\tV$ with mutation structure such that~$p_{ij} = \pi_j$ for~$1\leq i \neq j\leq K$ and~$p_{ii}=1-\sum_{j\not=i} \pi_j$. To distinguish from the~$\pi_i$, we write~$p_i:=1-p_{ii}$ for the chance that an individual with parent of type~$i$ is not of type~$i$. As in the previous section, we apply Theorem~\ref{THM3}, and  to define a stationary pair~$(\tW, \tW')$, let~$\tX$ be distributed as a stationary distribution of the chain
and let~$\tX'$ be a step in the chain from~$\tX$. 
Set~$\tW=\tX/N$ and~$\tW'=\tX'/N$. 

We first compute~$\IE[X_i'-X_i|\tX]$. 
The first thing to note is that we can decompose the number of individuals of
type~$i$ in the~$\tX'$-generation into those that have parent of type~$i$ and those that do not.  In particular,  if we denote by~$M_i$ the number of offspring in~$\tV$ that originate from a parent of type~$i$ in the~$\tX$-generation and write~$\tM=(M_1,\ldots, M_K)$, then
\ben{
\law(X_i' | \tM)=\law(Y_1(\tM)+Y_2(\tM)), \label{46}
} 
where~$Y_1(\tM)\sim\Bin(N-M_i, \pi_i)$,~$Y_2(\tM)\sim \Bin(M_i,1- p_i)$, 
and these two variables are independent given~$\tM$. From here we easily have
\be{
\IE(X_i'|\tX,\tM) = \pi_i(N-M_i) + (1-p_i) M_i.
}
Now noting the exchangeability of~$\tV$ implies~$\IE V_j = 1$, and hence~$\IE (M_i|\tX) = X_i$, take the expectation with respect to~$\tM$ to find
\bes{
\IE(X_i'|\tX) &= \pi_i(N-X_i) + (1-p_i)X_i\\
	&= \pi_i N + (1-\sigma)X_i,
}
where~$\sigma = \sum_{i=1}^K \pi_i$. If we now set~$\tW = \tX/N$ and~$\tW' = \tX' / N$ we find that
\be{
\IE[\tW' - \tW | \tW] = \boldsymbol{\pi}-\sigma\tW.
}
Recalling our definition of~$\alpha$ from~\eq{3} and letting 
\be{
\ta=\frac{2(N-1)}{\alpha} \boldsymbol{\pi},
} 
we are in the setting of Theorem~\ref{THM3} with 
$\Lambda=\frac{\alpha}{2(N-1)}\times \Id$ and~$\tR=0$.
As in Section~\ref{sec2}, applying the theorem is a relatively straightforward but 
tedious calculation involving conditioning and computing various moment formulas.
For the latter we record the following lemma.
\begin{lemma}\label{47}
If\/~$\tV$ is a Cannings exchangeable offspring vector, if
$\alpha$,~$\beta$, and~$\gamma$ are the moments defined at~\eq{3}, and 
$\delta:=\IE \klg{V_1(V_1-1)(V_1-2)(V_1-3)}$,
then
\ban{
\IE V_1^2 &= 1+\alpha, \label{48}\\
\IE V_1V_2 &=1 - \alpha\frac{1}{N-1},\label{49}\\
\IE V_1^3 &= 1+3\alpha+\beta, \label{50}\\
\IE V_1 V_2 V_3&=1-\alpha\frac{3}{N-1}+\beta\frac{2}{(N-1)(N-2)},  \label{51}\\
\IE V_1^2 V_2&= 1+\alpha \frac{N-3}{N-1}-\beta\frac{1}{N-1},\label{52} \\
\IE V_1^2 V_2^2 &= 1+ \alpha \frac{2N-5}{N-1} -\beta \frac{2}{N-1} + \gamma,\label{53}\\
\IE V_1^4&= 1+7\alpha + 6 \beta + \delta \label{54}, \\
\begin{split}\label{55}
\IE V_1V_2V_3V_4&=1-\alpha\frac{6}{N-1}+\beta\frac{8}{(N-1)(N-2)} \\
	&\qquad+\gamma\frac{3}{(N-2)(N-3)}-\delta\frac{3}{(N-1)(N-2)(N-3)}, 
	\end{split}\\
\IE V_1^2 V_2 V_3&= 1+\alpha\frac{N-6}{N-1}-\beta \frac{2N-8}{(N-1)(N-2)}-\gamma\frac{1}{N-2}+\delta\frac{1}{(N-1)(N-2)}, \label{56} \\
\IE V_1^3 V_2 &= 1+\alpha \frac{3N-7}{N-1} +\beta \frac{N-6}{N-1}-\delta \frac{1}{N-1}. \label{57} 
}
\end{lemma}
\begin{proof}
Since~$\sum_{i=1}^N V_1 =N$ and the~$V_i$'s are exchangeable, we have that~$\IE V_1=1$.
Thus~$\alpha=\IE V_1 (V_1-1)=\IE V_1^2 -1$ which is~\eq{48}. Note that similarly,
\be{
N=\IE V_1 (V_1+\cdots+V_N) = \IE V_1^2 + (N-1) \IE V_1 V_2 = \alpha+1 +(N-1)\IE V_1V_2,
}
and rearranging gives~\eq{49}.

For~\eq{50}, we have that
\be{
\IE V_1^3= \IE V_1(V_1-1)(V_1-2)+3\IE V_1^2 -2\IE V_1=\beta+3(\alpha+1)-2,
}
and further,
\ba{
N^2&=\IE V_1(V_1+\cdots + V_N)^2 =  \IE V_1^3  + 3(N-1)\IE V_1^2 V_2 + (N-1)(N-2) \IE V_1 V_2 V_3,\\
(\alpha+1)N&=\IE V_1^2(V_1+\cdots+V_N)=\IE V_1^3 + (N-1) \IE V_1^2 V_2.
}
Solving these two equations yields the expressions for~\eq{51} and~\eq{52}. 

Moving forward similarly, we have
\ba{
\IE V_1^2V_2^2&= \IE V_1(V_1-1)V_2(V_2-1)+2 \IE V_1^2V_2-\IE V_1 V_2, \\
\IE V_1^4&=\IE V_1(V_1-1)(V_1-2)(V_1-3)+6\IE V_1 (V_1-1)(V_1-2)+7\IE V_1(V_1-1)+\IE V_1,
}
and using previous expressions gives~\eq{53} and~\eq{54}. Along the same lines, we have
\ba{
\IE(V_1V_2V_3(V_1 + \cdots + V_N)) &= N \IE(V_1V_2V_3) = 3\IE(V_1^2V_2V_3) + (N-3) \IE(V_1V_2V_3V_4)\\
\IE(V_1^2V_2(V_1 + \cdots + V_N)) &= N\IE(V_1^2V_2) = \IE(V_1^3V_2) + \IE(V_1^2V_2^2) + (N-2)\IE(V_1^2V_2V_3)\\
\IE(V_1^3(V_1 + \cdots + V_N)) &= N \IE(V_1^3) = \IE(V_1^4) + (N-1)\IE(V_1^3V_2).
}
Plugging in values for known quantities in these three equations and solving yields~\eq{55},~\eq{56}, and~\eq{57}.
\end{proof}

We first work on the~$A_2$ term from Theorem~\ref{THM3} which only requires
two moments.

\begin{lemma}\label{58}
For~$\tV,\tX, \tM$ defined above,~$\alpha$ defined at~\eq{3}, and~$1\leq i\not=j\leq (K-1),$
\bes{
\IE(M_i|\tX) &= X_i,\\
\IE(M_i^2|\tX) &= X_i^2\left(1-\frac{\alpha}{N-1}\right) + X_i\frac{\alpha N}{N-1},\\
\IE(M_iM_j|\tX)& = X_iX_j\left(1 - \frac{\alpha}{N-1}\right).
}
\end{lemma}
\begin{proof}
Using exchangeability, without loss of generality,
\bes{ 
\IE(M_i|\tX) &= \IE[V_1 + \cdots + V_{X_i}|\tX] = X_i, \\
\IE(M_i^2|\tX) &= \IE[( V_1 + \cdots + V_{X_i})^2|\tX] = X_i \IE(V_1^2) + X_i(X_i-1)\IE(V_1V_2),\\
\IE(M_iM_j|\tX) &= \IE[ (V_1 + \cdots + V_{X_i})(V_{X_i+1} + \cdots + V_{X_i+X_j})|\tX]= X_i X_j \IE(V_1V_2),
}
The lemma now follows by using the formulas
for the moments of the~$V_i$ in Lemma~\ref{47}.
\end{proof}

\begin{lemma}\label{59}
For~$1\leq i\leq (K-1)$,~$(\tW, \tW'), \pi_i, p_i, \sigma$ defined above, and~$\alpha$
defined at~\eq{3},
\bes{
\IE[ (W_i' - W_i)^2 | \tW] &= W_i^2\left[\frac{-\alpha}{N-1} - \alpha \left(\frac{\sigma^2 - 2\sigma}{N-1}\right) + \sigma^2\right]\\
	&\ \ \ + W_i\left[ \frac{\alpha}{N-1} - \alpha \left(\frac{2\sigma -\sigma^2}{N-1}\right) + \frac{p_i(1-p_i) - \pi_i (1-\pi_i)}{N} - 2\pi_i\sigma\right]\\
	&\ \ \ + \pi_i(1-\pi_i)/N + \pi_i^2. 
}
\end{lemma}
\begin{proof}
Using the decomposition of~\eq{46},
\bes{
\IE[ (X_i' - X_i)^2 | \tX, \tM] &= (N-M_i)\pi_i(1-\pi_i) + (N-M_i)^2\pi_i^2 + M_i(1-p_i)p_i + M_i^2(1-p_i)^2\\
	&\ \ \ + 2(N-M_i)\pi_iM_i(1-p_i)- 2X_i(\pi_i N + (1-\sigma)M_i) + X_i^2\\
	&= M_i^2(1-\sigma)^2\\
	&\ \ \ + M_i[p_i(1-p_i)  -\pi_i(1-\pi_i) - 2N\pi_i^2 + 2N \pi_i(1-p_i)-2X_i(1-\sigma)]\\
	&\ \ \ + N\pi_i(1-\pi_i) + N^2 \pi_i^2 - 2N \pi_iX_i  + X_i^2. \\ 
}
Now taking expectation with respect to~$M_i$ using Lemma~\ref{58},
\bes{
\IE[ (X_i' - X_i)^2 | \tX] 
	&= \left[X_i^2\left(1-\frac{\alpha}{N-1}\right) + X_i\frac{\alpha N}{N-1}\right](1-\sigma)^2\\
	&\qquad + X_i[ p_i(1-p_i)  -\pi_i(1-\pi_i) - 2N\pi_i^2 + 2N\pi_i(1-p_i)-2X_i(1-\sigma)]\\
	&\qquad + N\pi_i(1-\pi_i) + N^2 \pi_i^2 - 2N \pi_iX_i  + X_i^2 \\
		&= X_i^2\left[1 + \left(1-\frac{\alpha}{N-1}\right)(1-\sigma)^2 - 2(1-\sigma)\right]\\
	&\ \ \ + X_i\left[ \frac{\alpha N}{N-1}(1-\sigma)^2- 2N\pi_ip_i + p_i(1-p_i) - \pi_i(1-\pi_i) -2N\pi_i^2  \right]\\
	&\ \ \ + N\pi_i(1-\pi_i) + N^2 \pi_i^2. 
}
Dividing this last expression by~$N^2$ and rearranging gives the lemma.
\end{proof}
\begin{lemma}\label{60}
For~$1\leq i\not=j \leq (K-1)$,~$(\tW, \tW'), \pi_i, \pi_j, p_i, p_j,\sigma$ defined above, and~$\alpha$
defined at~\eq{3},
\bes{
\IE[ (W_i'-W_i)(W_j'-W_j) | \tW ] &= W_iW_j\left[  -\frac{\alpha}{N-1} + \frac{\alpha\sigma(2-\sigma)}{N-1} + \sigma^2\right]\\
	&\qquad+\left(W_i\pi_j+W_j\pi_i\right)\left[ -\frac{1}{N} -\frac{N-1}{N}\sigma \right]+\frac{N-1}{N}\pi_i\pi_j.
}
\end{lemma}
\begin{proof}
Given~$\tM$, we can write~$(X_i', X_j', N-X_i'-X_j')$ 
as the sum of three independent multinomial random
variables corresponding to the counts of types~$i$,~$j$ and neither~$i$ or~$j$
in~$\tX'$ coming from 
individuals in the previous~$\tX$-generation having types~$i$,~$j$, and neither~$i$ or~$j$.
Then the parameters of these multinomials are
$M_i, (1-p_i, \pi_j, 1-p_i-\pi_j)$;
$M_j, (\pi_i,1-p_j, \pi_j, 1-p_j-\pi_i)$; and 
$N-M_i-M_j, (\pi_i, \pi_j, 1-\pi_i-\pi_j)$.
From this description and multinomial moment formulas (e.g., Lemma~\ref{lem10}),
it's straightforward to find that
\bes{
\IE[X_i'X_j'| \tX, \tM] &= M_i(M_i-1)(1-p_i)\pi_j + M_j(M_j-1)\pi_i(1-p_j) \\
	&\qquad + (N-M_i-M_j)(N-M_i-M_j-1)\pi_i\pi_j\\
	&\qquad+M_i(1-p_i)M_j(1-p_j) + M_i(1-p_i)(N-M_i-M_j) \pi_j \\
	&\qquad+ M_j\pi_iM_i\pi_j+ M_j\pi_i(N-M_i-M_j)\pi_j \\
	&\qquad+ (N-M_i-M_j)\pi_iM_i\pi_j + (N-M_i-M_j)\pi_i M_j (1-p_j)\\
	&=M_i M_j (1-\sigma)^2+(M_i \pi_j+M_j \pi_i)(1-\sigma)(N-1)+N(N-1)\pi_i \pi_j.
}
Also note that
\be{
\IE[X_i'X_j | \tX, \tM] = X_j[(N-M_i)\pi_i + (1-p_i)M_i],
}
so that these last two displays and Lemma~\ref{58} 
imply
\bes{
\IE[(X_i'-X_i)(X_j'-X_j)|\tX] &= X_iX_j\left[  (1-\sigma)^2\left(1-\frac{\alpha}{N-1}\right) - 2(1-\sigma) + 1\right]\\
	&\qquad+ ( X_i \pi_j+X_j\pi_i)[-(N-1) \sigma -1]+ N(N-1) \pi_i\pi_j.
}
Dividing this last expression by~$N^2$ and rearranging gives the lemma.
\end{proof}

The next lemma summarizes the bound on the~$A_2$ term of Theorem~\ref{THM3}
for this example.

\begin{lemma}\label{lem15}
For~$(\tW, \tW')$ and~$\sigma$ defined above and~$\alpha$ defined at~\eq{3}, 
if~$\lambda = \alpha/(2(N-1))$, then 
\bes{
\frac{1}{\lambda} \sum_{i,j=1}^{K-1} \IE& \left|\lambda W_i(\delta_{i j}-W_j)-\frac{1}{2}\IE[ (W_i'-W_i)(W_j'-W_j)|\tW] \right|\\
	&\leq  \sigma^2\left[ (K-1)^2 + \frac{N-1}{\alpha}\left( K^2 +1\right)\right]  + \sigma \left[ 2(K-1)^2 + \frac{3K-5}{\alpha} \right].
}
\end{lemma}
\begin{proof}
Using Lemmas~\ref{59} and~\ref{60},
\bes{
\frac{1}{\lambda} &\sum_{i,j=1}^{K-1} \IE \left|\lambda W_i(\delta_{i j}-W_j)-\frac{1}{2}\IE[ (W_i'-W_i)(W_j'-W_j)|\tW] \right|\\
&\leq \frac{N-1}{\alpha}\sum_{i=1}^{K-1} \left(\alpha \left(\frac{\sigma^2 + 2\sigma}{N-1}\right) + \sigma^2 + \frac{p_i(1-p_i) + \pi_i (1-\pi_i)}{N} + 2\pi_i\sigma+ \pi_i^2\right) \\
	&\ \ \ + \frac{N-1}{\alpha} \sum_{i \neq j}^{K-1} \left( \frac{\alpha}{N-1}(2\sigma + \sigma^2)  + \frac{\pi_i + \pi_j}{N} + \sigma^2 +\sigma\pi_i + \sigma\pi_j + \pi_i \pi_j\right)\\
	&\leq (K-1)(\sigma^2 + 2\sigma) +  \frac{N-1}{\alpha}[(K-1)\sigma^2 + (K-1)\sigma/N + 3\sigma^2 ]\\
	&\quad + (K-1)(K-2)(\sigma^2 + 2\sigma) + \frac{N-1}{\alpha}\left[\frac{ 2(K-2)\sigma}{N} + \sigma^2((K-1)(K-2)+2(K-2)+1) \right],\\	
	&\leq \sigma^2\left[ (K-1)^2 + \frac{N-1}{\alpha}\left( K^2 +1\right)\right]  + \sigma \left[ 2(K-1)^2 + \frac{3K-5}{\alpha} \right].\qedhere
}
\end{proof}

To compute the~$A_3$ term of Theorem~\ref{THM3}, we
need higher moment information.

\begin{lemma}\label{61}
For~$\tV,\tX, \tM$ defined above
and~$1\leq i\leq (K-1),$
\bes{
\IE(M_i^3|\tX) 
 &= X_i^3 [\IE(V_1V_2V_3)] + X_i^2[ 3\IE(V_1^2V_2) - 3\IE(V_1V_2V_3)]\\
	& \qquad +  X_i[ \IE(V_1^3) - 3\IE(V_1^2V_2) + 2\IE(V_1V_2V_3)]\\
\IE(M_i^4|\tX)
&= X_i^4 [\IE(V_1V_2V_3V_4)] + X_i^3[ 6\IE(V_1^2V_2V_3) - 6\IE(V_1V_2V_3V_4)]\\
	&\qquad + X_i^2 [ 4\IE(V_1^3V_2) + 3\IE(V_1^2V_2^2) - 18\IE(V_1^2V_2V_3) + 11\IE(V_1V_2V_3V_4)]\\
	&\qquad + X_i[ \IE(V_1^4) - 4\IE(V_1^3V_2) - 3\IE(V_1^2V_2^2) + 12\IE(V_1^2V_2V_3) - 6 \IE(V_1V_2V_3V_4)].
}
\end{lemma}
\begin{proof}
Similar to the proof of Lemma~\ref{47}, exchangeability implies
\bes{ 
\IE(M_i^3|\tX) &= \IE[(V_1 + \cdots + V_{X_i})^3|\tX],\\
	&= X_i \IE(V_1^3) + 3X_i(X_i-1)\IE(V_1^2V_2) + X_i(X_i-1)(X_i-2) \IE(V_1V_2V_3),\\
\IE(M_i^4|\tX) &= \IE[(V_1 + \cdots + V_{X_i})^4|\tX],\\
	&=X_i \IE(V_1)^4 + 4X_i(X_i-1)\IE(V_1^3V_2) + 3X_i(X_i-1)\IE(V_1^2V_2^2),\\
	&\qquad + 6X_i(X_i-1)(X_i-2) \IE(V_1^2V_2V_3) + X_i(X_i-1)(X_i-2)(X_i-3) \IE(V_1V_2V_3V_4).
}
The lemma now follows by rearranging these equations.
\end{proof}

\begin{lemma}\label{lem16}
For~$(\tW, \tW')$ and~$\sigma$ defined above and~$\alpha,\beta,\gamma, \delta$ defined at~\eq{3}
and~$N>1$, 
if~$\lambda = \alpha/(2(N-1))$, then 
\ba{
&\frac{1}{\lambda}\sum_{i,j,k=1}^{K-1}\IE |(W_i'-W_i)(W_j'-W_j)(W_k'-W_k)|\\
	&\quad \leq 2(K-1)^3\left(\left( \frac{3\sigma^2}{N\alpha} + \frac{\sigma}{N^2\alpha} \right)^{1/4} + \left( \frac{\rho}{N^3\alpha}\right)^{1/4} + \left(\frac{N\sigma^4}{\alpha}\right)^{1/4} \right)^2 \left(\sqrt{\frac\sigma\alpha}+ 1 + \sqrt{\frac{N\sigma^2}{\alpha}}\right).
}
where
\be{
\rho := 
\frac{  N^2\beta}{2(N-1)}+ \frac{(3 N^4) \gamma }{(N-2)(N-3)}+ \frac{(4N^4+3N^2)\delta}{(N-1)(N-2)(N-3)}.
}
\end{lemma}
\begin{proof}
Decompose
\ba{
X_i' - X_i &= [X_i' - (M_i(1-p_i) + (N-M_i)\pi_i)] + [(M_i-X_i)(1-\sigma)] + [N\pi_i - X_i \sigma]\\
	&=:E_i + F_i + G_i.
}
Using H\"older's inequality followed by Minkowski's inequality, we find
\ban{
\sum_{i,j,k}^{K-1}& \IE \abs{(W_i'-W_i)(W_j'-W_j)(W_k'-W_k)} \notag\\
	&\qquad= \frac{1}{N^3}\sum_{i,j,k}^{K-1}\IE \abs{(E_i +F_i+ G_i)(E_j +F_j+ G_j)(E_k +F_k+ G_k)}\notag \\
	&\qquad	\leq \frac{1}{N^3}\sum_{i,j,k}^{K-1}\left[\IE (E_i+F_i+G_i)^4 \IE (E_j+F_j+G_j)^4 \right]^{1/4}\left[\IE (E_k+F_k+G_k)^2\right]^{1/2}\notag\\
	\begin{split}\label{62}
	& \qquad \leq \frac{1}{N^3}\left(\sum_{i=1}^{K-1}\left[(\IE E_i^4)^{1/4}+(\IE F_i^4)^{1/4}+(\IE G_i^4)^{1/4} \right]\right)^2 \\
	&\qquad\qquad\qquad\qquad \times \sum_{k=1}^{K-1}\left[(\IE E_k^2)^{1/2}+(\IE F_k^2)^{1/2}+(\IE G_k^2)^{1/2}\right].
	\end{split}
}
Recall the decomposition~\eq{46} of~$\law(X_i'|\tM)=\law(Y_1(\tM)+Y_2(\tM))$ as a sum of
conditionally (on~$\tM$) 
independent binomials
and note that if~$Y\sim\Bin(n,p)$ then 
\bes{
\IE(Y-np)^4 &= 3(np(1-p))^2 + np(1-p)(1-6p(1-p)) \leq 3(np(1-p))^2 + np(1-p),
}
so that
\ba{
\IE[E_i^4|\tM]&= \IE[(Y_1(\tM) - \IE [Y_1(\tM)|\tM] + Y_2(\tM)- \IE [Y_2(\tM)|\tM])^4| \tM] \\
  &\leq  3 (M_ip_i(1-p_i))^2 + M_ip_i(1-p_i)	+ 6 M_i(1-p_i)p_i(N-M_i)\pi_i(1-\pi_i)\\
	& \ \ \ + 3((N-M_i)\pi_i(1-\pi_i))^2 + (N-M_i)\pi_i(1-\pi_i)\\
	&\leq 3(N(p_i(1-p_i) + \pi_i(1-\pi_i)))^2 + N(p_i(1-p_i) + \pi_i(1-\pi_i))\\
	&\leq 3(N \sigma)^2 + N\sigma.
}
Using a similar argument for the second moment, we thus have for all~$1\leq i \leq K-1$,
\ben{\label{63}
\IE E_i^2\leq N \sigma,  \hspace{1cm} \IE E_i^4\leq 3(N\sigma)^2 + N\sigma.
}
Now note that~$|G_i| \leq (N-X_i) \pi_i + X_i(\sigma-\pi_i) \leq N\sigma$, so that for all~$1\leq i \leq K-1$,
\ben{\label{64}
\IE G_i^2\leq (N\sigma)^2, \hspace{1cm} \IE G_i^4\leq  (N\sigma)^4.
}
For the~$F_i=M_i-X_i$ moments, first note that Lemma~\ref{58} implies 
\ben{\label{65}
\IE[F_i^2|\tX]=\IE[(M_i - X_i)^2|\tX]  
=\frac{\alpha X_i (N-X_i)}{N-1} \leq \frac{\alpha N^2}{N-1}.
}
Furthermore, using Lemmas~\ref{47},~\ref{58}, and~\ref{61},
\ba{
&\IE[(M_i - X_i)^4| \tX] = \IE(M_i^4 | \tX) - 4X_i \IE(M_i^3|\tX) + 6X_i^2 \IE(M_i^2|\tX) - 4X_i^3 \IE(M_i|X_i) + X_i^4\\
&\quad= X_i^4\left\{ \frac{3\gamma}{(N-2)(N-3)} + \frac{(-3)\delta}{(N-1)(N-2)(N-3)}\right\}\\
	&\qquad + X_i^3\left\{ \frac{-6 N\gamma}{(N-2)(N-3)} + \frac{6N \delta}{(N-1)(N-2)(N-3)}\right\}\\ 
	&\qquad + X_i^2 \left\{ \frac{-\alpha}{N-1}  + \frac{(-2N+4)\beta}{(N-1)(N-2)} + \frac{(3N^2 + 3N -3)\gamma}{(N-2)(N-3)}  + \frac{(-4N^2 + 2N - 3)\delta}{(N-1)(N-2)(N-3)} \right\}\\
	&\qquad + X_i \left\{ \frac{(-5N + 6)\alpha}{N-1} + \frac{(2N^2 -4N)\beta}{(N-1)(N-2)}+ \frac{(-3N^2+3N)\gamma}{(N-2)(N-3)} + \frac{(N^3 - 2N^2 + 3N)\delta}{(N-1)(N-2)(N-3)}\right\}.
}
Now using that~$0\leq X_i \leq N$ (and assuming~$N > 1$),
we have 
\ba{
&\alpha X_i (-X_i -5N+6) \leq 0,  &  &\beta X_i((-2N+4)X_i+(2N^2-4N))\leq \beta N^2(N-2)/2, \\
& 3\gamma X_i^3(X_i-2N)\leq 0, & &3\gamma X_i((N^2 + N -1)X_i-N^2+N)\leq 3\gamma N^4, \\
& 3 \delta X_i^3(-X_i+2N)\leq 3\delta N^4, & & \delta[X_i^2(-4N^2 + 2N - 3)+X_i(N^3 - 2N^2 + 3N)]\leq \delta(N^4+3N^2).
}
Combining these inequalities with the previous display, we have
\ben{\label{66}
\IE F_i^4=\IE[(M_i - X_i)^4] \leq \frac{  N^2\beta}{2(N-1)}+ \frac{(3 N^4) \gamma }{(N-2)(N-3)}+ \frac{(4N^4+3N^2)\delta}{(N-1)(N-2)(N-3)}= \rho.
}
Now using the inequalities~\eq{63},~\eq{64},~\eq{65}, and~\eq{66} in~\eq{62} yields
the lemma.
\end{proof}

\begin{proof}[Proof of Theorem~\ref{THM2}] We apply Theorem~\ref{THM3} with~$\Lambda = \frac{\alpha}{2(N-1)}\times \Id$.
From Lemmas~\ref{lem15} for~$A_2$ 
and~\ref{lem16} for~$A_3$ we obtain
\ba{
A_2&\leq \sigma^2\left[ (K-1)^2 + \frac{N-1}{\alpha}\left( K^2 +1\right)\right]  + \sigma \left[ 2(K-1)^2 + \frac{3K-5}{\alpha} \right],\\
A_3	&\leq 2(K-1)^3\left(\left( \frac{3\sigma^2}{N\alpha} + \frac{\sigma}{N^2\alpha} \right)^{1/4} + \left( \frac{\rho}{N^3\alpha}\right)^{1/4} + \left(\frac{N\sigma^4}{\alpha}\right)^{1/4} \right)^2 \left(\sqrt{\frac\sigma\alpha}+ 1 + \sqrt{\frac{N\sigma^2}{\alpha}}\right),
}
where
\be{
	\rho := 
\frac{  N^2\beta}{2(N-1)}+ \frac{(3 N^4) \gamma }{(N-2)(N-3)}+ \frac{(4N^4+3N^2)\delta}{(N-1)(N-2)(N-3)}.
}
The final bound in Theorem~\ref{THM2} is now obtained through straightforward manipulations and applying some standard analytic inequalities,
in particular,~$\sigma=\eta(\alpha/N)$ and~$\delta\leq (N-3)\beta$.
\end{proof}

\section*{Acknowledgments}

We thank the anonymous referee for helpful comments and for pointing out an omission in an earlier version of the manuscript (proof of existence of partial derivatives of the solution to the Stein equation).
NR received support from ARC grant DP150101459; AR received support from NUS Research Grant R-155-000-124-112. This work was done
partially while the authors were visiting the Institute for Mathematical Sciences, National University of Singapore in 2015. The visit was supported by the Institute. HG would also like to thank the School of Mathematics at the University of Melbourne for their hospitality while some of this work was done.

%

\end{document}